\documentclass[a4paper,reqno]{amsart}

\usepackage{amsmath,amsthm,amssymb,amsfonts,mathrsfs}
\usepackage{esint}
\usepackage[usenames]{color}
\usepackage{graphicx}

\theoremstyle{plain}
\newtheorem{theorem}{Theorem}[section]
\newtheorem{corollary}[theorem]{Corollary}

\newtheorem{lemma}[theorem]{Lemma}
\newtheorem{proposition}[theorem]{Proposition}
\theoremstyle{definition}

\theoremstyle{remark}
\newtheorem{remark}{Remark}[section]
\numberwithin{equation}{section}

\newcommand{\RR}{\mathbb{R}}

\newcommand{\tr}{\,{\rm tr}}

\newcommand{\diag}{\,{\rm diag}}
\newcommand{\sgn}{\,{\rm sgn}}

\newcommand{\ep}{\varepsilon}

\newcommand{\W}{{\rm W}}
\newcommand{\LL}{{\rm L}}

\newcommand{\na}{\nabla}
\newcommand{\kk}{k}
\newcommand{\dx}{\,{\rm d}x}
\newcommand{\bm}{\lambda}

\mathchardef\emptyset="001F

\begin{document}

\title[Narrow ribbons of nematic elastomers]
{Shape programming for narrow ribbons of \\nematic elastomers}

\author[V.~Agostiniani]{Virginia Agostiniani}
\address{SISSA, via Bonomea 265, 34136 Trieste - Italy} 
\email{vagostin@sissa.it}

\author[A.~DeSimone]{Antonio DeSimone}
\address{SISSA, via Bonomea 265, 34136 Trieste - Italy}
\email{desimone@sissa.it}

\author[K.~Koumatos]
{Konstantinos Koumatos}
\address{Gran Sasso Science Institute,
viale Francesco Crispi 7,
67100 L'Aquila - Italy}
\email{konstantinos.koumatos@gssi.infn.it}

\thanks{}

\begin{abstract} 

Using the theory of $\Gamma$-convergence, 
we derive from three-dimension-al elasticity 
new one-dimensional models for non-Euclidean elastic ribbons,
i.e. ribbons exhibiting spontaneous curvature and twist.
We apply the models to shape-selection problems for 
thin films of nematic elastomers with twist and 
splay-bend texture of the nematic director.
For the former, we discuss
the possibility of helicoid-like shapes
as an alternative to spiral ribbons.
 
\end{abstract}

\maketitle


\section{Introduction}
\label{intro}

Shape morphing systems are common in Biology.
They are used to control locomotion in 
unicellular organisms \cite{Arroyo1,Arroyo2}
and to produce controlled motions in plants 
\cite{Dawson,Fratzl,Godinho1,Godinho2,Erodium}.
Differential swelling and shrinkage processes, 
partially hindered by fibers, lead to dynamical 
conformation changes which are essential in the life of
many botanical systems \cite{Armon}.
Inspired by Nature, many attempts have been reported in the recent
literature to synthesize artificial shape-morphing systems
based on synthetic soft materials \cite{Studart,Reyssat,Kim}
and the interest in these phenomena is steadily growing. 

A useful tool has emerged in the mathematical literature
to describe the mechanics of shape programming,
namely, \emph{non-Euclidean} structures
(non-Euclidean plates and rods).
These are elastic structures described by functionals
which are minimised when exhibiting
nonzero curvature. 
The relevant energy functionals are often postulated on the
basis of physical intuition \cite{Kle_Efr_Sha_2007}, 
but in more recent attempts they are derived from 
three dimensional models
\cite{Schmidt2007,Lew_Pak2011,Ag_De_bend},
through rigorous dimension reduction
based on the theory of $\Gamma$-convergence,
following the approach pioneered
in \cite{F_J_M_2002}.

Liquid crystal elastomers provide an interesting model system
for the study of shape programming. They are polymeric materials 
that respond to external stimuli 
(temperature, light, electric fields)
by changing shape 
\cite{Warner_book,ferro,AgDe2,AgDalDe2,Ag_Bl_Ko}
and are typically manufactured
as thin films 
\cite{Bla_Ter_War,DeSim_Dolz,Aha_Sha_Kup,CCDD1,CDD2}.
Suitable textures of the nematic director
imprinted at fabrication lead to thin structures 
with tunable and controllable spontaneous curvature,
see \cite{hybrid,Sawa2011,Urayama2013}.
In particular, for the \emph{twist} geometry
(nematic director always parallel to the
mid-plane of the film and rotating by $\pi/2$ from
the bottom to the top surface of the film)
it has been observed 
both experimentally and computationally
\cite{Sawa2011,Teresi_Varano} that, 
depending on the aspect ratio of the mid-plane,
either spiral ribbons 
(this is the case of large width over length aspect ratio)
or helicoid-like shapes
(this is the case of small width over length aspect ratio)
emerge spontaneously.  

In this paper, we provide a rigorous mathematical description
of thin structures of nematic elastomers
(in the case of twist and splay-bend geometry)
where the minimisers of the deduced energy functionals
reproduce the experimentally observed minimum energy configurations.
Our analysis stems from the combination of two main results:
first, we use the $3$D-to-$2$D dimension reduction result
in \cite{Ag_De_bend} for (non-Euclidean) thin films
of nematic elastomers, and, secondly, 
we use a non-Euclidean version of the $2$D-to-$1$D dimension reduction 
result of \cite{Freddi2015}, 
where a corrected version of the
well-known Sadowsky functional 
is derived for the mechanical description of 
inextensible elastic ribbons.
The reader is also referred to \cite{Kirby2014,Efrati2014}
and all other papers in the same
special issue of the Journal of Elasticity for more
material on the mechanics of elastic ribbons.

Our results show that the technique of rigorous dimensional
reduction based on $\Gamma$-convergence,
far from being just a mathematical exercise,
can provide a tool to \emph{derive},
rather than \emph{postulate},
the functional form and the material parameters
(elastic constants, spontaneous curvature and twist, etc...)
for dimensionally reduced models
of thin structures.
We concentrate our discussion on liquid
crystal elastomers, but clearly our method
is applicable to more general systems,
whenever differential spontaneous distortions 
in the cross section induce spontaneous curvature and twist
of the mid-line of the rod.

The starting point of the subsequent analysis is 
a family of non-Euclidean plate models defined on a narrow strip
of width $\ep$ cut out from the plane and forming an angle 
$\theta$ with the horizontal axis (see \eqref{physical} below). 
In Section \ref{theory} we set-up our 2D model and show
that minor modifications of the results of \cite{Freddi2015} 
allow us to derive in the limit as $\ep\downarrow0$
the 1D model defined in \eqref{limit_J_theta}--\eqref{Qasmin}.
Again in Section \ref{theory}, we observe that examples of
our starting 2D theory are given by twist and splay-bend nematic
elastomer sheets, as obtained from
3D nonlinear elasticity in \cite{Ag_De_bend}.  
The limiting 1D theory, which is a non-Euclidean rod theory,
is then explicitly computed for twist and splay-bend 
nematic elastomers in Section \ref{twist} and 
\ref{splay-bend}, respectively.
In particular, the explicit expression of the limiting 
energy densities associated with the rods through 
their flexural strains around the width axis and the 
torsional strains, and the corresponding minimisers, 
are given in Proposition
\ref{prop_T} and Lemma \ref{minimi_T} (in the twist case)
and in Proposition \ref{prop:QS} and Lemma \ref{minimi_S}
(in the splay-bend case).

\section{A non-Euclidean Sadowsky functional}
\label{theory}

Let $\omega$ be an open planar domain of $\RR^2$.
In the framework of a nonlinear plate theory \cite{F_J_M_2002},
we consider the bending energy
\begin{equation}
\label{eq:bendingenergy}
\int_{\omega}\Big\{c_1|A_{\hat v}(\hat z) - \bar{A}|^2
+
c_2\tr^2(A_{\hat v}(\hat z)-\bar A)
+\bar e\Big\}{\rm d}\hat z 
\end{equation}
associated with a developable surface $\hat v(\omega)$,
$\hat v$ being a deformation from $\omega$ to $\RR^3$.
In the previous expression,
$A_{\hat v}(\cdot)\in\RR^{2\times2}_{\rm sym}$ 
denotes the second fundamental form
of $\hat v(\omega)$, 
$c_1>0$ and $c_2>0$
are material constants, and $\bar A\in\RR^{2\times2}_{\rm sym}$ 
and $\bar e$ 
represent a characteristic target curvature tensor
and a characteristic nonnegative energy constant, respectively.
Moreover, the notation $\tr^2A$ stands for the square
of the trace of $A$.
We recall that $A_{\hat v}$ can be expressed as
$(\na\hat v)^{\rm T}\na\hat\nu$, where 
$\hat\nu=\partial_{z_1}\hat v\wedge\partial_{z_2}\hat v$.
In \cite{Ag_De_bend}, the two-dimensional energy 
\eqref{eq:bendingenergy} has been rigorously derived from
a three dimensional model for thin films of nematic elastomers
with \emph{splay-bend} and \emph{twist} orientation
of the nematic directors along the thickness
and the following explicit formulas have been obtained for $\bar A$ 
\begin{equation}
\label{barA}
\bar{A}_S\,=\,\kk\diag(-1,0), 
\qquad\ 
\bar{A}_T\,=\,\kk\diag(-1,1),
\qquad\quad
\kk:=\frac{6\,\eta_0}{\pi^2h_0},
\end{equation}
and for $\bar e$
\begin{equation}
\label{bare}
\bar e_S
\,=\, 
\mu\,(1+\bm)
\left(\frac{\pi^4-12}{32}\right)
\frac{\eta_0^2}{h_0^2}
\qquad\ 
\bar e_T
\,=\,
\mu\left(\frac{\pi^4 - 4\pi^2 - 48}{8\pi^4}\right)
\frac{\eta_0^2}{h_0^2}.
\end{equation} 
Here and throughout the paper we use the indices  
``\,$S$\,'' and ``\,$T$\,'' for the 
quantities related to the splay-bend and the twist case,
respectively.
In the previous formulas,
$\eta_0$ is a positive dimensionless parameter quantifying
the magnitude of the spontaneous strain variation 
along the (small) thickness $h_0$ of the film,
$\mu$ is the elastic shear modulus, 
and $\bm+2\mu/3$ is the bulk modulus.
Also, the material constants appearing in \eqref{eq:bendingenergy}
are given by  $c_1=\mu/12$ and $c_2=\bm\mu/12$.
We refer the reader to \cite{Ag_De_bend} for a 
detailed description of the three-dimensional model 
and of the splay-bend and twist nematic director fields.
In Sections \ref{twist} and \ref{splay-bend} we specialize
our results to the case where the curvature tensor 
$\bar A$ is of the form \eqref{barA},
while in the rest of this section we focus on 
a general energy density
of type \eqref{eq:bendingenergy}.
  
We cut out of the planar region $\omega$
a narrow strip
\begin{equation*}
S_{\ep}^{\theta}:=
\Big\{z_1\mathsf e_1^{\theta}+z_2\mathsf e_2^{\theta}:
z_1\in(-\ell/2,\ell/2),\ z_2\in(-\ep/2,\ep/2)\Big\}
\subset\omega,
\qquad0\leq\theta<\pi,
\end{equation*}
with
\[
\mathsf e_i^{\theta} = R_\theta\mathsf e_i,\quad i=1,2,
\qquad
R_\theta = \left( \begin{array}{rr} \cos\theta & -\sin\theta \\ \sin\theta & \cos\theta \end{array}\right)
\]
and consider the energy \eqref{eq:bendingenergy}
restricted to the strip $S_{\ep}^{\theta}$, namely
\begin{equation}
\label{physical}
\hat{\mathscr E}_{\ep}^{\theta}(\hat v)
\,:=\,
\int_{S_{\ep}^{\theta}}
\Big\{c_1|A_{\hat v}(\hat z) - \bar A\,|^2+
c_2\tr^2(A_{\hat v}(\hat z)-\bar A\,)+\bar e\Big\}\,{\rm d}\hat z,
\end{equation}
where $\hat v:S_{\ep}^{\theta}\to\RR^3$ is a deformation such
that $\hat v(S_{\ep}^{\theta})$ 
is a developable surface.
We are interested in examining the behaviour of the minimisers of the 
functionals $\hat{\mathscr E}_\ep^{\theta}$ 
in the limit of vanishing width, i.e. $\ep\downarrow0$.
Notice that using the function
$v:S_{\ep}\to\RR^3$ defined in the unrotated strip 
$S_{\ep}:=S_{\ep}^0$ as $v(z)=\hat v(R_\theta z)$,
we have that $v(S_{\ep})$ is developable and
$\hat{\mathscr E}_{\ep}^{\theta}(\hat v)$ can be rewritten as
\begin{align*}
\hat{\mathscr E}_{\ep}^{\theta}(\hat v)
&\,=\,
\int_{S_{\ep}}
\Big\{c_1|A_{\hat v}(R_\theta z) - \bar A\,|^2+
c_2\tr^2(A_{\hat v}(z)-\bar A\,)+\bar e\Big\}\,{\rm d}z\\
&\,=\,
\int_{S_{\ep}}
\Big\{c_1|A_v(z) - R_\theta^{\rm T}\bar{A}R_\theta|^2+
c_2\tr^2(A_v(z)-\bar A\,)+\bar e\Big\}
{\rm d}z,
\end{align*}
where in the second equality we have used the fact that
$A_{\hat v}(R_\theta z)=R_\theta A_v(z)R_\theta^{\rm T}$.
Now, introducing a suitable rescaling and setting 
\begin{equation}
\label{rescaled_energy}
\mathscr E_{\ep}^{\theta}(v)
\,:=\,\frac1{\ep}
\int_{S_{\ep}}
\Big\{c_1|\,A_v(z) - \bar A^{\theta}\,|^2+
c_2\tr^2(A_v(z)-\bar A\,)+\bar e\Big\}\,{\rm d}z,
\end{equation}
with
\begin{equation}
\label{A_theta}
\bar A^{\theta}\,:=\,R_\theta^{\rm T}\bar AR_\theta,
\end{equation}
we have that
$\hat{\mathscr E}_{\ep}^{\theta}(\hat v)
=
\ep\,
\mathscr E_{\ep}^{\theta}(v).
$
Having this identification in mind, from now on
we always deal with the functional 
$v\mapsto\mathscr E_{\ep}^{\theta}(v)$.
Expanding the integrand we obtain the following general form of the bending energy
\begin{equation}
\label{eq:bendingenergy1}
\mathscr E_{\ep}^{\theta}(v)
=
\frac1{\ep}\int_{S_\ep}\Big\{
c |A_v(z)|^2 + L^\theta(A_v(z))\Big\}\,{\rm d}z,
\end{equation}
where
\begin{equation}
\label{L_theta}
L^\theta(A_v) 
\,:=\, - 2c_1 A_v\cdot\bar{A}^\theta -2c_2\tr A_v\tr\bar{A}^\theta
+ c_1|\bar{A}^\theta|^2 + c_2\tr^2\bar{A}^\theta + \bar e,
\end{equation}
and the notation is meant to remind the reader that, 
as a function of the matrix $A$, $L^\theta$ is linear.
In the above expression, we have set $c=c_1+c_2$ and we have made use of the fact that
\[
\tr^2A = |A|^2 + 2\det A,
\]
for all $A\in\RR^{2\times2}_{\rm sym}$. 
In fact, since $v$ is an isometry of the planar strip $S_{\ep}$, 
the Gaussian curvature associated with $v(S_{\ep})$ vanishes, i.e.
\[
\det A_v(z) = 0.
\]
The natural function space for $v$ is the space of $\W^{2,2}$ isometries of $S_\ep$ defined as
\[
\W^{2,2}_{\rm iso}(S_\ep,\RR^3)\,:=\,
\Big\{v\in \W^{2,2}(S_\ep,\RR^3)\,:\, \partial_i v\cdot\partial_j v = \delta_{ij}\Big\}.
\]
In order to express the energy over the fixed domain
\[
S=I\times\Big(-\frac12,\frac12\Big),
\qquad\qquad
I := \Big(-\frac\ell2,\frac\ell2\Big),
\]
we change variables and define the rescaled version 
$y:S\to\RR^3$ of $v$, given by
\[
y(x_1,x_2) = v(x_1,\ep x_2).
\]
The following procedure is rather standard and we use the notation 
of \cite{Freddi2015} as, in the sequel, our proofs will be largely based on this paper. 
By introducing the scaled gradient
\[
\nabla_\ep \cdot = (\partial_1 \cdot |\ep^{-1}\partial_2\cdot)
\]
we obtain that $\nabla_\ep y(x_1,x_2) = \nabla v(x_1,\ep x_2)$ and $y$ belongs to the space of scaled isometries of $S$ defined as
\[
\W^{2,2}_{{\rm iso},\ep}(S,\RR^3)
\,:=\,
\Big\{y\in \W^{2,2}(S,\RR^3)\,:\, 
|\partial_1 y| = |\ep^{-1}\partial_2 y|=1,\,\,
\partial_1 y\cdot\partial_2 y = 0\mbox{ a.e. in S} 
\Big\}.
\]
Similarly, we may define the scaled unit normal to $y(S)$ by
\[
n_{y,\ep} = \partial_1 y\wedge\ep^{-1}\partial_2 y
\]
and the scaled second fundamental form associated to $y(S)$ by
\[
A_{y,\ep} =  \left( \begin{array}{rrr} n_{y,\ep}\cdot \partial_1\partial_1 y & \, & \ep^{-1}n_{y,\ep}\cdot \partial_1\partial_2 y \\  \ep^{-1}n_{y,\ep}\cdot \partial_1\partial_2 y &\, &  \ep^{-2}n_{y,\ep}\cdot \partial_2\partial_2 y \end{array}\right).
\]
With this definition, $A_{y,\ep}(x_1,x_2) = A_v(x_1,\ep x_2)$ and 
$\mathscr{E}_\ep^{\theta}(v) 
= 
\mathscr{J}_\ep^\theta(y)$, 
where the functional
\begin{equation}
\label{our_fun}
\mathscr J_\ep^\theta(y) 
:= 
\int_S \Big\{c |A_{y,\ep}|^2 + L^\theta(A_{y,\ep})\Big\}\dx
\end{equation}
is defined over the space $\W^{2,2}_{{\rm iso},\ep}(S,\RR^3)$
of scaled isometries of $S$.

\begin{lemma}[Compactness]
\label{lemma:compactness}
Suppose $(y_\ep)\subset \W^{2,2}_{{\rm iso},\ep}(S,\RR^3)$ satisfy
\[
\sup_\ep \mathscr J_\ep^\theta(y_\ep) < \infty.
\]
Then, up to a subsequence and additive constants, there exist a deformation 
$y\in \W^{2,2}(I,\RR^3)$ and a
orthonormal frame
$(d_1|d_2|d_3)\in \W^{1,2}(I,{\rm SO}(3))$ 
fulfilling
\[
d_1 = y^\prime\quad\mbox{ and }\quad d_1^\prime \cdot d_2 = 0
\quad\mbox{ a.e. in } I,
\]
and such that
\[
y_\ep\rightharpoonup y\quad\mbox{in}\quad\W^{2,2}(S,\RR^3),
\qquad\qquad\nabla_\ep y_\ep \rightharpoonup (d_1|d_2)
\quad\mbox{in}\quad\W^{1,2}(S,\RR^{3\times2}).
\]
Moreover, for some $\gamma\in \LL^2(S,\RR^3)$,
\[
A_{y,\ep}\rightharpoonup 
\left( 
\begin{array}{rr} d_1^\prime\cdot d_3 & d_2^\prime\cdot d_3  \\ d_2^\prime\cdot d_3  & \gamma  
\end{array}
\right)
\quad\mbox{in}\quad\LL^2(S,\RR^{2\times2}_{\rm sym}).
\]
\end{lemma}

\begin{proof}
Note that
\begin{align*}
\mathscr{J}_\ep^\theta(y_\ep) 
&= \int_S 
\Big\{c_1 |A_{y_\ep,\ep} - \bar{A}^\theta\,|^2 
+c_2 \tr^2\left(A_{y_\ep,\ep}(x) - \bar A^\theta\,\right) 
+ \bar e\Big\}\dx \\
& \geq \int_S c_1 | A_{y_\ep,\ep}(x)-\bar{A}^\theta\,|^2\dx.
\end{align*}
But, since $\bar{A}^\theta$ is constant, this implies that 
$\|A_{y_\ep,\ep}\|^2_{\LL^2(S,\RR^3)}$ is bounded uniformly in $\ep$
and the proof is identical to the proof of Lemma 2.1 
in \cite{Freddi2015}.
\end{proof}

In order to state the $\Gamma$-convergence result, we define
\begin{equation}
\label{funct_class}
\mathcal{A} \,:=\, 
\Big\{(d_1,d_2,d_3)\,:\,(d_1|d_2|d_3)\in \W^{1,2}(I,{\rm SO}(3)),\, d_1^\prime\cdot d_2 = 0\mbox{ a.e. in } I
\Big\},
\end{equation}
and the functional $\mathscr{J}^\theta:\mathcal{A}\to \RR$ by
\begin{equation}
\label{limit_J_theta}
\mathscr{J}^\theta(d_1,d_2,d_3)
\,:=\,
\int_I \overline Q^{\,\theta}(d_1^\prime\cdot d_3,d_2^\prime\cdot d_3)\dx_1,
\end{equation}
where
\begin{equation}
\label{Qasmin}
\overline Q^{\,\theta}(\alpha,\beta):=\min_{\gamma\in\RR}\left\{c |M|^2 + 2c|\det M| + L^\theta(M)\,:\,M = \left( \begin{array}{rr} \alpha & \beta  \\ \beta  & \gamma  \end{array}\right)\right\},
\end{equation}
and $L^\theta(M)$ is defined according to \eqref{L_theta}.
We recall that the constraint $d_1^\prime\cdot d_2=0$
means that the narrow strip does not bend within its plane or, equivalently,
that there is no flexure around the direction of $d_3$.

\begin{theorem}[$\Gamma$-convergence]
\label{thm:gamma}
The functionals $\mathscr J_\ep^\theta$ $\Gamma$-converge to 
$\mathscr J^\theta$ as $\ep\to0$ in the following sense:
\begin{itemize}
\item[(1)] {\rm (}$\Gamma$-$\liminf$ inequality{\rm )} 
for every sequence $(y_\ep)\subset \W^{2,2}_{{\rm iso},\ep}(S,\RR^3)$, $y\in W^{2,2}(I,\RR^3)$ and $(d_1,d_2,d_3)\in\mathcal{A}$ with $y^\prime = d_1$ a.e. in $I$, $y_\ep\rightharpoonup y$ in $\W^{2,2}(S,\RR^3)$ and $\nabla_\ep y_\ep \rightharpoonup (d_1|d_2)$ in $\W^{1,2}(S,\RR^{3\times2})$, 
we have 
\[
\liminf_{\ep\to0}\mathscr J_\ep^\theta(y_\ep)\geq 
\mathscr J^\theta(d_1,d_2,d_3);
\]
\item[(2)] {\rm (}recovery sequence{\rm )}
for every $(d_1,d_2,d_3)\in\mathcal{A}$ there exists $(y_\ep)\subset \W^{2,2}_{{\rm iso}, \ep}(S,\RR^3)$ and (up to an additive constant) $y$ satisfying $y^\prime=d_1$ 
such that $y_\ep\rightharpoonup y$ in $\W^{2,2}(S,\RR^3)$, $\nabla_\ep y_\ep\rightharpoonup (d_1|d_2)$ in $\W^{1,2}(S,\RR^{3\times2})$, and
\[
\lim_{\ep\to0} \mathscr J_\ep^\theta(y_\ep) 
= 
\mathscr J^\theta(d_1,d_2,d_3).
\]
\end{itemize}
\end{theorem}

In proving the existence of a recovery sequence, we will need the following lemma which is a slight variation of 
\cite[Lemma 3.1]{Freddi2015}.

\begin{lemma}
\label{lemma:det}
For every $M\in\LL^2(I,\RR^{2\times2}_{\rm sym})$ there exists a sequence 
$(M_n)\subset\LL^2(I,\RR^{2\times2}_{\rm sym})$ satisfying $\det M_n = 0$ a.e. in $I$ and for all $n\in\mathbb{N}$ such that $M_n\rightharpoonup M$ in $\LL^2(I,\RR^{2\times2}_{\rm sym})$ and
\[
\int_I 
\left[c |M_n|^2 +L^\theta(M_n) \right]{\rm d}x_1 
\longrightarrow 
\int_I \left[c |M|^2 +2c |\det M| + L^\theta(M)\right]{\rm d}x_1.
\]
\end{lemma}

\begin{proof}
Without loss of generality, we may assume that $\det M \neq 0$. We first consider the case $M$ constant and diagonal, i.e.
\[
M = \lambda_1 e_1\otimes e_1 + \lambda_2 e_2\otimes e_2.
\]
Set
\[
\mu = \frac{|\lambda_1|}{|\lambda_1| + |\lambda_2|} \in (0,1)
\]
to find that
\[
|M|^2 + 2|\det M| = \lambda_1^2+\lambda_2^2+2|\lambda_1\lambda_2| = \frac{\lambda_1^2}{\mu} +  \frac{\lambda_2^2}{1-\mu}.
\]
Denoting by $\chi:\RR\to\{0,1\}$ the 1-periodic extension of the characteristic function of the interval $(0,\mu)$ and defining $M_n:I\to\RR^{2\times2}_{\rm sym}$ by
\[
M_n(x_1) = \chi(nx_1)\frac{\lambda_1}{\mu}e_1\otimes e_1 + (1-\chi(nx_1))\frac{\lambda_2}{1-\mu}e_2\otimes e_2
\]
we infer that $\det M_n = 0$ and $M_n\overset{\ast}{\rightharpoonup}M$ in $\LL^\infty(I,\RR^{2\times2}_{\rm sym})$, since $\chi(n\cdot)\overset{\ast}{\rightharpoonup}\theta$ in $\LL^\infty(I)$. Also,
\begin{align*}
\int_I c |M_n|^2 + L^\theta(M_n) \dx_1 & = \int_I c\frac{\lambda_1^2}{\mu} +  c\frac{\lambda_2^2}{1-\mu} + L^\theta(M_n) \dx_1\\
& = \int_I  c|M|^2 + 2c |\det M| + L^\theta(M_n)  \dx_1 \\ 
&\to \int_I  c|M|^2 + 2c |\det M| + L^\theta(M)  \dx_1
\end{align*}
as $n\to\infty$, since $L^\theta$ is affine. If $M\in\RR^{2\times2}_{\rm sym}$ is not diagonal, yet constant, we can find an orthogonal matrix $R$ such that $R^{\rm T}M R$ is diagonal and apply the above argument.

If $M$ is instead piecewise constant, we apply the same argument to each interval on which $M$ is constant and, for general $M\in\LL^2(I,\RR^{2\times2}_{\rm sym})$ we approximate $M$ in the strong topology 
of $\LL^2(I,\RR^{2\times2}_{\rm sym})$ by a sequence $(M_k)$ of piecewise constant maps. We may then apply the above argument to each $M_k$ to obtain a sequence $(M_{k,n})$ with the required properties and such that
\[
||M_{k,n}||_{L^2}^2 = \int_I  |M_{k,n}|^2\dx_1 = \int_I  |M_k|^2 + 2 |\det M_k|\dx_1 \leq 2 ||M_k||_{L^2}^2,
\]
i.e. the sequence $(M_{k,n})$ is bounded in $\LL^2$. Noting that the weak topology in $\LL^2$ is metrisable on bounded sets and taking a diagonal sequence, we conclude the proof.
\end{proof}

\begin{proof}[Proof of Theorem \ref{thm:gamma}]
\quad\\
\noindent (1) ($\Gamma$-$\liminf$ inequality) Let $(y_\ep)\subset \W^{2,2}_{{\rm iso},\ep}(S,\RR^3)$, $y\in W^{2,2}(I,\RR^3)$ and $(d_1,d_2,d_3)\in\mathcal{A}$ with $y^\prime = d_1$ a.e. in $I$, $y_\ep\rightharpoonup y$ in $\W^{2,2}(S,\RR^3)$ and $\nabla_\ep y_\ep \rightharpoonup (d_1|d_2)$ in $\W^{1,2}(S,\RR^3)$. We may assume that $\liminf_\ep \mathscr{J}^\theta_\ep(y_\ep)<\infty$ as, otherwise the result follows trivially, and by passing to a subsequence that $\sup_\ep \mathscr{J}_\ep^\theta(y_\ep)<\infty$.

Lemma \ref{lemma:compactness} now states that $A_{y_\ep,\ep}\rightharpoonup A$ in $\LL^2(S,\RR^{2\times2}_{\rm sym})$ where
\[
A = \left( \begin{array}{rr} d_1^\prime\cdot d_3 & d_2^\prime\cdot d_3  \\ d_2^\prime\cdot d_3  & \gamma  \end{array}\right).
\]
Set $A^\ep:=A_{y_\ep,\ep}$ and recall that for every 2$\times$2, symmetric matrix $M$, it holds that $\tr^2M = |M|^2 + 2\det M$. Since $\det A^\ep = 0$ we infer that $|A^\ep|^2 = \tr^2A^{\ep}$. Also note that
\[
\tr^2 A^\ep = \left(A^\ep_{11} - A^\ep_{22}\right)^2 + 4 A^\ep_{11}A^\ep_{22} =  \left(A^\ep_{11} - A^\ep_{22}\right)^2 + 4\left(A^\ep_{12}\right)^2.
\]
Splitting $S = S^+ \cup S^-$ where $S^+:=\{x\in S\,:\,\det A(x)\geq 0\}$ and $S^-:=\{x\in S\,:\,\det A(x)< 0\}$, we find that
\begin{align*}
\liminf_\ep \mathscr{J}_\ep^\theta(y_\ep) & = \liminf_\ep \left\{ \int_{S^+} c\tr^2A^\ep \dx + \int_{S^-} c\left(A^\ep_{11} - A^\ep_{22}\right)^2 + 4c\left(A^\ep_{12}\right)^2 \dx\right\}\\
&\quad + \lim_\ep \int_S L^\theta(A^\ep) \dx \\
&\geq \int_{S^+} c\tr^2A \dx + \int_{S^-} c\left(A_{11} - A_{22}\right)^2 + 4c\left(A_{12}\right)^2 \dx + \int_S L^\theta(A) \dx
\end{align*}
since, with respect to the weak $\LL^2$ topology, the first two functionals are lower semicontinuous by convexity and the third functional is continuous by linearity. Next, note that
\[
\left(A_{11} - A_{22}\right)^2 + 4\left(A_{12}\right)^2 = |A|^2 - 2\det A
\]
so that
\begin{align*}
\liminf_\ep \mathscr{J}_\ep^\theta(y_\ep) 
&\geq 
c\int_{S^+}\left(\,|A|^2 +2\det A\right) \dx 
+ 
c \int_{S^-}\left(\,|A|^2 - 2\det A\right) \dx \\
&\quad + \int_S L^\theta(A) \dx\\
& = 
\int_S 
\big[c|A|^2 +2c|\det A| + L^\theta(A)\big] \dx\\
&\geq \mathscr{J}^\theta(d_1,d_2,d_3)
\end{align*}

\noindent (2) (recovery sequence) Fix $(d_1,d_2,d_3)\in\mathcal{A}$ and $y\in\W^{2,2}(I,\RR^3)$ such that $y^\prime=d_1$ a.e. in $I$. Define $R := (y^\prime|d_2|d_3)\in{\rm SO}(3)$ a.e. in $I$ and
\[
M = \left(\begin{array}{cc} y^{\prime\prime}\cdot d_3 & d_2^{\prime}\cdot d_3 \\ d_2^{\prime}\cdot d_3 & \gamma\end{array}\right)
\]
where $\gamma\in\LL^2(I)$ is chosen such that
\[
\overline Q^{\,\theta}(y^{\prime\prime}\cdot d_3,d_2^{\prime}\cdot d_3) = \overline Q^{\,\theta}(M_{11},M_{12}) = c |M|^2 +2c |\det M| + L^\theta(M).
\]
Through Lemma \ref{lemma:det} we find a sequence 
$(M_n)\subset\LL^2(I,\RR^{2\times2}_{\rm sym})$ such that 
$\det M_n = 0$, $M_n\rightharpoonup M$ in 
$\LL^2(I,\RR^{2\times2}_{\rm sym})$ and, as $n\to\infty$,
\[
\int_I 
\left[c |M_n|^2 + L^\theta(M_n)\right] \dx_1 
\longrightarrow 
\int_I 
\left[c |M|^2 +2c |\det M| + L^\theta(M)\right]\dx_1.
\]
We now proceed exactly as in \cite{Freddi2015}. Let $\lambda_n = \tr M_n$. Since $M_n$ is symmetric and $\det M_n = 0$, we may find $\beta_n(x_1)\in(-\pi/2,\pi/2]$ such that
\[
M_n = \left(\begin{array}{cc} \cos\beta_n & -\sin\beta_n \\ \sin\beta_n & \cos\beta_n \end{array}\right)\left(\begin{array}{cc} \lambda_n & 0 \\ 0 & 0 \end{array}\right)\left(\begin{array}{cc} \cos\beta_n & \sin\beta_n \\ -\sin\beta_n & \cos\beta_n \end{array}\right),
\]
which is well-defined by setting $\beta_n(x_1) = 0$ when $\lambda_n(x_1) = 0$. As in \cite{Freddi2015}, we may assume that $\lambda_n\in\LL^\infty(I)$ and, by mollifying, we find $\lambda_{n,k}\in C^\infty(\overline{I})$ and $\beta_{n,k}\in C^\infty(\overline{I})$ with the property that
\begin{itemize}
\item $|\beta_{n,k}|<\pi/2$ (note the strict inequality) for all $x_1\in\overline{I}$;
\item $\lambda_{n,k}\to\lambda_n$ in $\LL^p(I)$ as $k\to\infty$ and for all $p<\infty$;
\item $\beta_{n,k}\to\beta_n$ in $\LL^p(I)$ as $k\to\infty$ and for all $p<\infty$.
\end{itemize}
Next, set
\[
M_{n,k} = \left(\begin{array}{cc} \cos\beta_{n,k} & -\sin\beta_{n,k} \\ \sin\beta_{n,k} & \cos\beta_{n,k} \end{array}\right)\left(\begin{array}{cc} \lambda_{n,k} & 0 \\ 0 & 0 \end{array}\right)\left(\begin{array}{cc} \cos\beta_{n,k} & \sin\beta_{n,k} \\ -\sin\beta_{n,k} & \cos\beta_{n,k} \end{array}\right),
\]
and note that $\det M_{n,k} = 0$ for all $k$, $n$ and $M_{n,k}\to M_n$ in $\LL^2(I,\RR^{2\times2}_{\rm sym})$ as $k\to\infty$. Thus, by extracting a diagonal sequence, we find $\lambda^j$, $\beta^j\in C^\infty(\overline{I})$ with $|\beta^j|<\pi/2$ on $\overline{I}$ and for 
\begin{align*}
M^j & =  \left(\begin{array}{cc} \cos\beta^j & -\sin\beta^j \\ \sin\beta^j & \cos\beta^j \end{array}\right)\left(\begin{array}{cc} \lambda^j & 0 \\ 0 & 0 \end{array}\right)\left(\begin{array}{cc} \cos\beta^j & \sin\beta^j \\ -\sin\beta^j & \cos\beta^j \end{array}\right)\\
& =  \lambda^j \left(\begin{array}{cc} \cos^2\beta^j & \sin\beta^j\cos\beta^j \\ \sin\beta^j\cos\beta^j & \sin^2\beta^j \end{array}\right)
\end{align*}
it holds that $\det M^j = 0$ for all $j$, and as $j\to\infty$, $M^j\rightharpoonup M$ in $\LL^2(I,\RR^{2\times2}_{\rm sym})$ as well as
\[
\int_I
\left[ 
c |M^j|^2 + L^\theta(M^j)\right] \dx_1 
\longrightarrow 
\int_I 
\left[
c |M|^2 +2c |\det M| + L^\theta(M)\right] \dx_1.
\]
Extend $\beta^j$ smoothly to $\RR$ maintaining the constraint $|\beta^j|<\pi/2$ and for $t^j:=\pi/2 + \beta^j$, define
\[
\tilde{b}^j(\xi_1) = \cos t^j(\xi_1)e_1 + \sin t^j(\xi_1)e_2\mbox{ and } \Phi^j(\xi_1,\xi_2) = \xi_1 e_1 + \xi_2\tilde{b}^j(\xi_1),
\]
noting that there exists $\ep^j$ such that for all $\ep\leq\ep^j$ the map $(\Phi^j)^{-1}:S_\ep\to\RR^2$ is well defined 
(see \cite{Freddi2015}).

Define $R^j:I\to{\rm SO}(3)$ as the solution to the ODE
\[
\left(R^j\right)^\prime = R^j  \left(\begin{array}{ccc} 0 & 0 & -M_{11}^j \\ 0 & 0 & -M_{12}^j \\ M_{11}^j & M_{12}^j & 0\end{array}\right)
\]
with initial data $R^j(0) = R(0) = (y^\prime(0)|d_2(0)|d_3(0))$ and set
\[
d_k^j(s) = R^j(s) e_k,\,\,k=1,2,3\mbox{ and } y^j(s) = y(0) + \int_0^s d^j_1.
\]
Then $y^j\rightharpoonup y$ in $\W^{2,2}(I,\RR^3)$ and by the ODE we infer that
\begin{align*}
& (d_1^j)^\prime\cdot d_2^j = 0,\\
& (d_2^j)^\prime\cdot d_3^j = M_{12}^j = -\lambda^j\sin t^j\cos t^j, \\
& (d_1^j)^\prime\cdot d_3^j = M_{11}^j = \lambda^j\sin^2t^j.
\end{align*}
Define
\begin{align*}
b^j(\xi_1) & = \cos t^j(\xi_1) d_1^j(\xi_1) + \sin t^j(\xi_1) d_2^j(\xi_1),\\
v^j(\xi_1,\xi_2) & = y^j(\xi_1) + \xi_2 b^j(\xi_1),\\
u^j(x_1,x_2) & = v^j\left(\left(\Phi^j\right)^{-1}(x_1,x_2)\right).
\end{align*}
Following \cite{Freddi2015}, one then obtains that $(\nabla u^j)(\nabla u^j) = \mathbb{I}$, $\nabla u^j(\cdot, 0) = (d_1^j|d_2^j)$ and that $A_{u^j}(\cdot,0) = M^j(\cdot)$. For $\ep>0$ small enough, the maps $y^j_\ep: S\to\RR^3$ defined by $y^j_\ep(x_1,x_2) = u^j(x_1,\ep x_2)$ are well-defined scaled $C^2$ isometries of $S$ such that
\[
\nabla_\ep y^j_\ep \to \nabla u^j(\cdot,0) = (d_1^j|d_2^j) \mbox{ strongly in }\W^{1,2}(S,\RR^{3\times2})\mbox{ as }\ep\to0
\]
and since $A_{u^j}(x_1,0) = M^j(x_1)$, we also get that
\[
A_{y^j_\ep,\ep}\to M^j\mbox{ strongly in }\LL^2(S,\RR^{2\times2}_{\rm sym})\mbox{ as }\ep\to0.
\]
Hence, we find that
\begin{align*}
\lim_{\ep\to0} \mathscr{J}_\ep^\theta(y^j_\ep) 
& = 
\lim_{\ep\to0}\int_S 
\left[c|A_{y^j_\ep,\ep}|^2 + L^\theta(A_{y^j_\ep,\ep})\right] 
\dx\\
& = 
\int_S \left[c|M^j|^2 + L^\theta(M^j)\right]\dx \\
& \overset{j\to\infty}{\longrightarrow} 
\int_S \left[c|M|^2 + 2c|\det M| + L^\theta(M)\right]\dx\\
& = \int_S\overline Q^{\,\theta}(d_1^\prime\cdot d_3,d_2^\prime\cdot d_3)\dx_1 = \mathscr{J}^\theta(d_1,d_2,d_3).
\end{align*}
The proof can then be finished by taking diagonal sequences.
\end{proof}

\section{The twist case}
\label{twist}

In the case where the (physical) energy \eqref{physical}
is derived from a three-dimensional model 
with a twist-type nematic director field imprinted
in the thickness of an elastomeric thin film,
the characteristic quantities 
$\bar A$ and $\bar e$ present in \eqref{physical}
are given by $\bar A_T$ and $\bar e_T$
in formulas \eqref{barA} and \eqref{bare}. 
Correspondingly, the $\theta$-dependent target curvature tensor
$\bar A^\theta$ defined in \eqref{A_theta} becomes  
\begin{equation}
\label{A^T}
\bar{A}_T^\theta 
\,=\,
\kk \left( \begin{array}{rr} -\cos 2\theta & \sin2\theta \\ \sin2\theta & \cos2\theta \end{array}\right) 
\,=:\, 
\kk\left( \begin{array}{rr} -a_\theta & b_\theta \\ 
b_\theta & a_\theta \end{array}\right),
\qquad\quad\theta\in[0,\pi),
\end{equation}
and $\tr \bar A_T^\theta =\tr\bar A_T =0$.
Moreover, the functional $\mathscr E_{\ep}^{\theta}$ defined
in \eqref{eq:bendingenergy1} in this case reads
\begin{equation*}
\mathscr E_{\ep,T}^{\theta}(v)
\,:=\,
\frac1{\ep}\int_{S_\ep}\Big\{c |A_v(z)|^2 + 
L_T^{\theta}(A_v(z))\Big\}{\rm d}z,
\end{equation*}
with
\[
L_T^{\theta}(A_v) 
\,:=\, 
-2c_1\,A_v\cdot\bar A_T^{\theta} 
+ 2c_1\kk^2 + \bar e_T.
\]
Also, the energy $\mathscr J_\ep^\theta$ defined
in the rescaled configuration $S$ (see \eqref{our_fun})
is
\begin{equation}
\label{ep_funct_T}
\mathscr J_{\ep,T}^{\theta}(y) 
:= 
\int_S \Big\{c |A_{y,\ep}(x)|^2 + L_T^\theta(A_{y,\ep}(x))\Big\}\dx,
\end{equation}
for every $y\in\W^{2,2}_{{\rm iso},\ep}(S,\RR^3)$.
We recall that
$\hat{\mathscr E}_{\ep}^{\theta}(\hat v)
=
\ep\,
\mathscr E_{\ep}^{\theta}(v)
=
\ep\,
\mathscr J_{\ep}^{\theta}(y),
$
where $\hat v:S_{\ep}^{\theta}\to\RR^3$,
$v:S_{\ep}\to\RR^3$,
$v:S\to\RR^3$ are isometries and are related to each other via
the following relations
\[
v(z)=\hat v(R_\theta z),
\qquad\qquad
y(x_1,x_2)=v(x_1,\ep x_2).
\]
Lemma \ref{lemma:compactness} and Theorem \ref{thm:gamma}
apply in particular for the functionals \eqref{ep_funct_T}.
As an easy consequence of the compactness and 
the $\Gamma$-convergence results, 
via standard arguments of the theory of $\Gamma$-convergence,
the following corollary holds.
In order to state it, we define
$\mathscr{J}_T^{\theta}:\mathcal{A}\to \RR$ as
\begin{equation}
\label{funct_lim_T}
\mathscr J_T^{\theta}(d_1,d_2,d_3)
\,:=\,
\int_I 
\overline Q_T^{\,\theta}
(d_1^\prime\cdot d_3,d_2^\prime\cdot d_3)\dx_1,
\end{equation}
where $\mathcal A$ is the class of orthonormal frames defined
in \eqref{funct_class} and $\overline Q_T^{\,\theta}$
is defined as in \eqref{Qasmin} with 
$L_T^\theta$ in place of $L^\theta$ (see also above).

\begin{corollary}
\label{cor:twist}
If $(y_\ep)\subset \W^{2,2}_{{\rm iso},\ep}(S,\RR^3)$
is a sequence of minimisers of $\mathscr J_{\ep,T}^{\theta}$,
then, up to a subsequence, we have that there exist
$y\in \W^{2,2}(I,\RR^3)$ and 
a minimiser $(d_1|d_2|d_3)\in\mathcal A$ of
$\mathscr J_T^{\theta}$ with $d_1 = y^\prime$ 
such that 
\[
y_\ep\rightharpoonup y\ \mbox{ in }\ \W^{2,2}(S,\RR^3),
\qquad
\nabla_\ep y_\ep \rightharpoonup (d_1|d_2)
\ \mbox{ in }\ \W^{1,2}(S,\RR^{3\times2}),
\]
and
\begin{equation}
\label{here_gamma}
A_{y,\ep}\rightharpoonup \left( \begin{array}{rr} d_1^\prime\cdot d_3 & d_2^\prime\cdot d_3  \\ d_2^\prime\cdot d_3  & \gamma  \end{array}\right)\ \mbox{ in }\ \LL^2(S,\RR^{2\times2}_{\rm sym}),
\qquad\mbox{for some }\ \gamma\in \LL^2(S,\RR^3).
\end{equation}
Moreover,
\begin{equation}
\label{cvg_minima_T}
\min_{\W^{2,2}_{{\rm iso},\ep}(S,\RR^3)}
\mathscr J_{\ep,T}^{\theta}
\longrightarrow
\min_\mathcal A
\mathscr J_T^{\theta}.
\end{equation}
\end{corollary}

Notice that thanks to \cite[Lemma 3.8]{Ag_De_bend}
and Lemma \ref{minimi_T} below, 
the minimum of $\mathscr J_{\ep,T}^{\theta}$
in $\W^{2,2}_{{\rm iso},\ep}(S,\RR^3)$
and the minimum of  
$\mathscr J_T^{\theta}$ in $\mathcal A$
can be computed explicitly,
so that the convergence in \eqref{cvg_minima_T} 
can be checked by hand.
The minimizing sequences and the minimizer of 
$\mathscr J_T^{\theta}$ can be computed
as well, together with $\gamma$ in \eqref{here_gamma}
(see the proof of Theorem \ref{thm:gamma} (2)).
Therefore, also the convergence of the minimizing
sequences can be checked by hand. 

The following proposition gives the
explicit expression of $\overline Q_T^{\,\theta}$.  

\begin{proposition}
\label{prop_T}
$\overline Q_T^{\,\theta}$ is a continuous function given by
\[
\overline Q_T^{\,\theta}(\alpha,\beta)
\,=\, 
\left\{\begin{array}{lc}
4c_1\kk\big(a_\theta\alpha - b_\theta\beta\big)
+c_1\kk^2\Big(2-\frac{c_1}ca_\theta^2\Big) + \bar e_T, 
& 
\quad\mbox{in}\quad\mathcal D_T\\
 4\big(c\beta^2 - c_1\kk b_\theta\beta\big)
+c_1\kk^2\Big(2-\frac{c_1}ca_\theta^2\Big) +\bar e_T,  
&
\quad\mbox{in}\quad\mathcal U_T\\
c\frac{(\alpha^2+\beta^2)^2}{\alpha^2} + 
2c_1\kk\Big(a_\theta\frac{\alpha^2-\beta^2}{\alpha} 
- 2b_\theta\beta\Big)
+2c_1\kk^2+\bar e_T,\quad  
& 
\quad\mbox{in}\quad\mathcal V_T, 
\end{array}\right.
\]
where $a_\theta$ and $b_\theta$ are defined in \eqref{A^T},
and
\begin{align*}
\mathcal D_T
&:=\left\{(\alpha,\beta)\in\RR^2:\frac{c_1}c\kk
a_\theta\,\alpha > \beta^2+\alpha^2\right\},\\
\mathcal U_T
&:=\left\{(\alpha,\beta)\in\RR^2:\frac{c_1}c\kk 
a_\theta\,\alpha \leq \beta^2-\alpha^2\right\},\\
\mathcal V_T
&:=\RR^2\setminus(\mathcal D_T\cup\mathcal U_T).
\end{align*}
\end{proposition}

\begin{proof}
For a matrix
\[
M = \left(\begin{array}{rr} \alpha & \beta \\ \beta & \gamma\end{array}\right),
\]
the expression 
for $\overline Q_T^{\,\theta}$ in \eqref{Qasmin} 
(with $L_T^\theta$ in place of $L^\theta$) becomes
\[
\overline Q_T^{\,\theta}(\alpha,\beta)=\min_{\gamma\in\RR} f(\gamma),
\]
where
\[
f(\gamma) 
:= 
c(\alpha^2 + 2\beta^2 +\gamma^2) + 2c|\alpha\gamma - \beta^2| 
+ 2c_1 \kk a_\theta (\alpha - \gamma) 
- 4c_1 \kk b_\theta\beta + 2c_1\kk^2 + \bar e_T.
\]
Note that if $\alpha=0$, $f$ reduces to the differentiable function
\[
f(\gamma) =  4c \beta^2 +c\gamma^2 - 2c_1 \kk a_\theta \gamma 
- 4c_1 \kk b_\theta\beta + 2c_1\kk^2 + \bar e_T
\]
and it is minimised at $\gamma = \frac{c_1}{c}\kk a_\theta$, 
i.e. for all $\beta\in\RR$
\[
\overline Q_T^{\,\theta}(0,\beta) 
=
 4\left(c\beta^2 - c_1\kk b_\theta\beta\right)
+c_1\kk^2\Big(2-\frac{c_1}c a_\theta^2\Big) +\bar e_T.
\]
Next assume that $\alpha\neq 0$. If $\gamma=\beta^2/\alpha$, 
\[
f(\beta^2/\alpha) 
=
c\frac{(\alpha^2+\beta^2)^2}{\alpha^2} + 
2c_1\kk\Big(a_\theta\frac{\alpha^2-\beta^2}{\alpha} 
- 2 b_\theta\beta\Big)
+2c_1\kk^2+\bar e_T
\] 
and for any $\gamma\neq\beta^2/\alpha$, the function $f$ is differentiable with
\[
f^\prime(\gamma)/2 = c\gamma + c\alpha\sgn(\alpha\gamma-\beta^2) 
- c_1\kk a_\theta,
\]
which vanishes at 
\[
\gamma = \frac{c_1}{c}\kk a_\theta - 
\alpha\sgn(\alpha\gamma-\beta^2).
\] 
If $\alpha\gamma > \beta^2$,
then $\gamma_1 = \frac{c_1}{c}\kk a_\theta - \alpha$ is the critical point and this can only be true in the regime
\[
\frac{c_1}c\kk
a_\theta\,\alpha > \beta^2+\alpha^2.
\]
In this case, we compute
\[
f(\gamma_1) 
= 
4c_1\kk(a_\theta\alpha - b_\theta\beta)
+c_1\kk^2\Big(2-\frac{c_1}c a_\theta^2\Big) + \bar e_T.
\]
Similarly, for $\alpha\gamma < \beta^2$, we find that 
$\gamma_2 = \frac{c_1}{c}\kk a_\theta + \alpha$ is the critical point which can only be true in the regime
\[
\frac{c_1}c\kk
a_\theta\,\alpha
< \beta^2-\alpha^2
\]
and then
\[
f(\gamma_2) 
=
 4\left(c\beta^2 - c_1\kk b_\theta\beta\right)
+c_1\kk^2\Big(2-\frac{c_1}c a_\theta^2\Big) +\bar e_T.
\]
On the other hand, in the regime
\[
\beta^2-\alpha^2
\leq 
\frac{c_1}c\kk
a_\theta\,\alpha
\leq 
\beta^2+\alpha^2,
\]
a straightforward computation shows that $f^\prime(\gamma)<0$ if $\gamma<\beta^2/\alpha$ and $f^\prime(\gamma)>0$ if $\gamma>\beta^2/\alpha$. Hence, in this regime, and with $\alpha\neq 0$, the minimum value of $f$ is achieved at $\gamma=\beta^2/\alpha$ so that
\[
\overline Q_T^{\,\theta}(\alpha,\beta) 
= 
f(\beta^2/\alpha) 
= 
c\frac{(\alpha^2+\beta^2)^2}{\alpha^2} + 
2c_1\kk\Big(\bar{a}_\theta\frac{\alpha^2-\beta^2}{\alpha} 
- 2 b_\theta\beta\Big)
+2c_1\kk^2+\bar e_T
\]
To compute $\overline Q_T^{\,\theta}$
in the respective regimes of values of 
$\kk a_\theta\,\alpha$,
one needs to understand whether the value of 
$f$ at its respective local minima $\gamma_1$ and $\gamma_2$ 
is lower than $f(\beta^2/\alpha)$. We compute
\begin{align*}
f(\gamma_1) - f(\beta^2/\alpha) 
& = 
- \frac{c_1^2}{c}\kk^2a_\theta^2 
+ 4c_1\kk a_\theta\alpha 
- c\frac{(\alpha^2+\beta^2)^2}{\alpha^2} - 2c_1\kk a_\theta\frac{\alpha^2-\beta^2}{\alpha}\\
& = - \frac{c_1^2}{c}\kk^2a_\theta^2 - 
c\frac{(\alpha^2+\beta^2)^2}{\alpha^2} 
+ 2c_1\kk a_\theta\frac{\alpha^2+\beta^2}{\alpha}\\
& = -c\left\{\frac{c_1^2}{c^2}\kk^2a_\theta^2 
+ \frac{(\alpha^2+\beta^2)^2}{\alpha^2} 
- 2\frac{c_1}{c}\kk a_\theta
\frac{\alpha^2+\beta^2}{\alpha}\right\}\\
& = -c\left[\frac{\alpha^2+\beta^2}{\alpha} - 
\frac{c_1}{c}\kk a_\theta\right]^2 \leq 0
\end{align*}
with equality if and only if
\[
\frac{c_1}c\kk
a_\theta\,\alpha
 = 
 \beta^2+\alpha^2.
\]
Similarly,
\begin{align*}
f(\gamma_2) - f(\beta^2/\alpha) 
& = 
- \frac{c_1^2}{c}\kk^2a_\theta^2 + 4c\beta^2 - 
c\frac{(\alpha^2+\beta^2)^2}{\alpha^2} - 2c_1\kk a_\theta\frac{\alpha^2-\beta^2}{\alpha}\\
& = - \frac{c_1^2}{c}\kk^2a_\theta^2 - 2c_1\kk a_\theta
\frac{\alpha^2-\beta^2}{\alpha} 
- c\frac{(\alpha^2-\beta^2)^2}{\alpha^2}\\
& = -c\left\{\frac{c_1^2}{c^2}\kk^2a_\theta^2 
+ \frac{(\alpha^2 - \beta^2)^2}{\alpha^2} 
+ 2\frac{c_1}{c}\kk a_\theta
\frac{\alpha^2-\beta^2}{\alpha}\right\}\\
& = -c\left[\frac{\alpha^2-\beta^2}{\alpha} + 
\frac{c_1}{c}\kk a_\theta\right]^2 \leq 0
\end{align*}
with equality if and only if
\[
\frac{c_1}c\kk
a_\theta\,\alpha
 = 
 \beta^2 - \alpha^2.
\]
Hence, we deduce the result.
Note that these computations show that
$\overline Q_T^{\,\theta}$ is continuous. 
\end{proof}

Notice that when $\kk=0$ (and $c=1$), modulo the constant $\bar e_T$ 
the expression for $\overline Q_T^{\,\theta}$
in Proposition \ref{prop_T}
reduces to expression (1.5) in \cite{Freddi2015}, namely,
\begin{equation}
\label{bar_Q}
\overline Q(\alpha,\beta)
:= 
\left\{\begin{array}{ll}
 4\,\beta^2
&
\quad\mbox{if}\quad\alpha^2\leq\beta^2\\
\frac{(\alpha^2+\beta^2)^2}{\alpha^2} 
& 
\quad\mbox{if}\quad\alpha^2>\beta^2.
\end{array}\right.
\end{equation}
Indeed, in this case
\begin{equation}
\label{sets_k_null}
\mathcal D_T=\emptyset,
\qquad\qquad
\mathcal U_T
=\left\{(\alpha,\beta)\in\RR^2:\alpha^2 \leq \beta^2\right\}.
\end{equation}
For $\kk>0$,
it is natural to distinguish the case $\theta=\pi/4$ 
(and, similarly, the case $\theta=3\pi/4$)
from all the other cases. Indeed, we have
$a_{\pi/4}=0$ and $b_{\pi/4}=1$, so that
$\mathcal D_T$ and $\mathcal U_T$
are again given by \eqref{sets_k_null}, whereas
\[
\overline Q_T^{\,\pi/4}(\alpha,\beta)
= 
\left\{\begin{array}{lc}
 4\big(c\beta^2 - c_1\kk\beta\big)
+2c_1\kk^2 +\bar e_T,  
&
\quad\mbox{if}\quad\alpha^2\leq\beta^2\\
c\frac{(\alpha^2+\beta^2)^2}{\alpha^2}
-4c_1\kk\beta
+2c_1\kk^2
+\bar e_T,
\quad  
& 
\quad\mbox{if}\quad\alpha^2>\beta^2.
\end{array}\right.
\]
Observe that for all
$\theta\in[0,\pi/2)\setminus\{\pi/4,3\pi/4\}$,
setting $\rho:=c_1ka_\theta/(2c)$,
we have that 
$\mathcal D_T$ coincides with the (open) disk 
$(\alpha-\rho)^2+\beta^2<\rho^2$
and $\mathcal U_T$ with the (closed) region 
inside the hyperbola $(\alpha+\rho)^2-\beta^2=\rho^2$.

\begin{figure}[htbp]
\label{figure_T}
\begin{center}
\includegraphics[width=4.1cm]{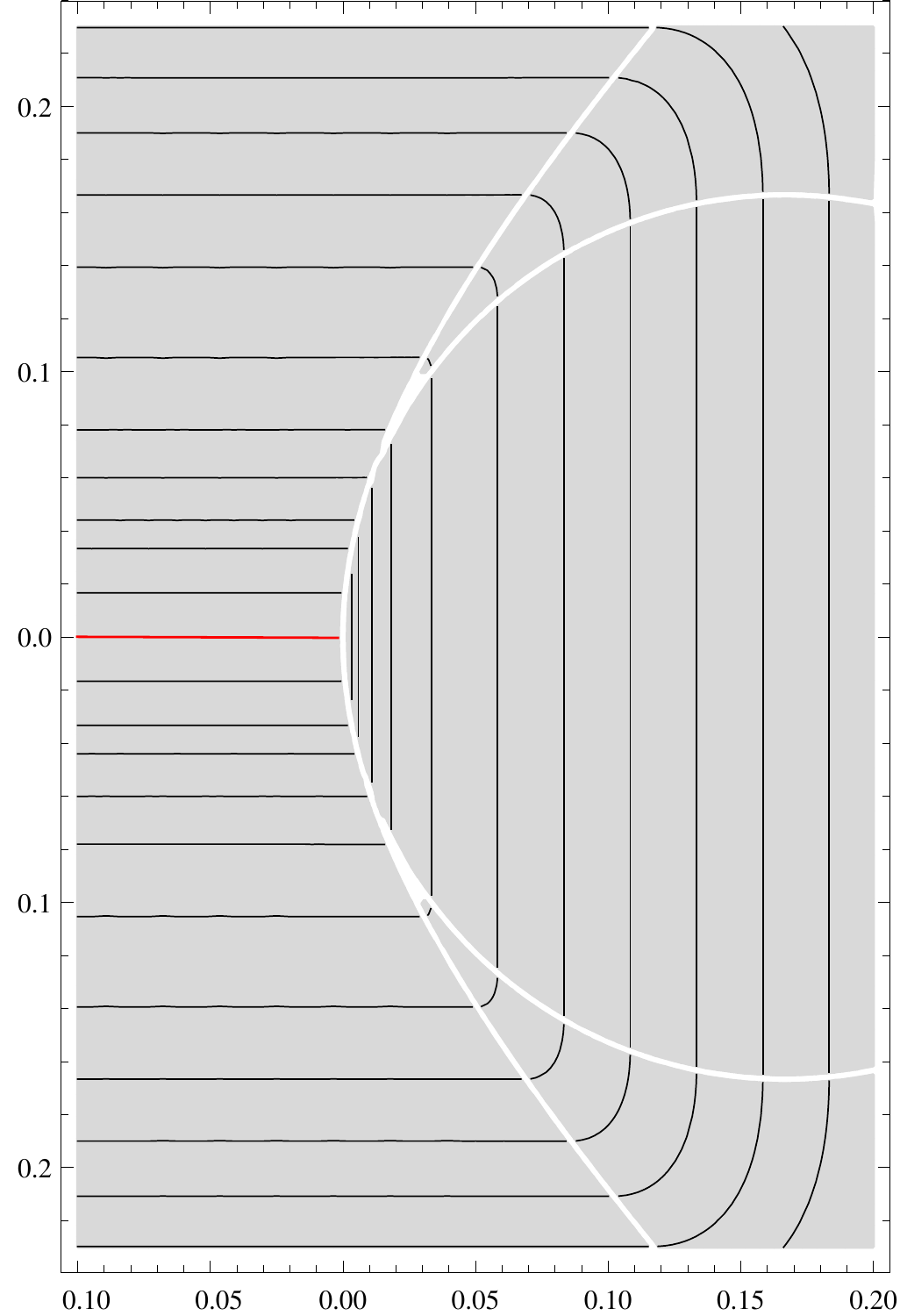}
\hspace{1cm}
\includegraphics[width=4.1cm]{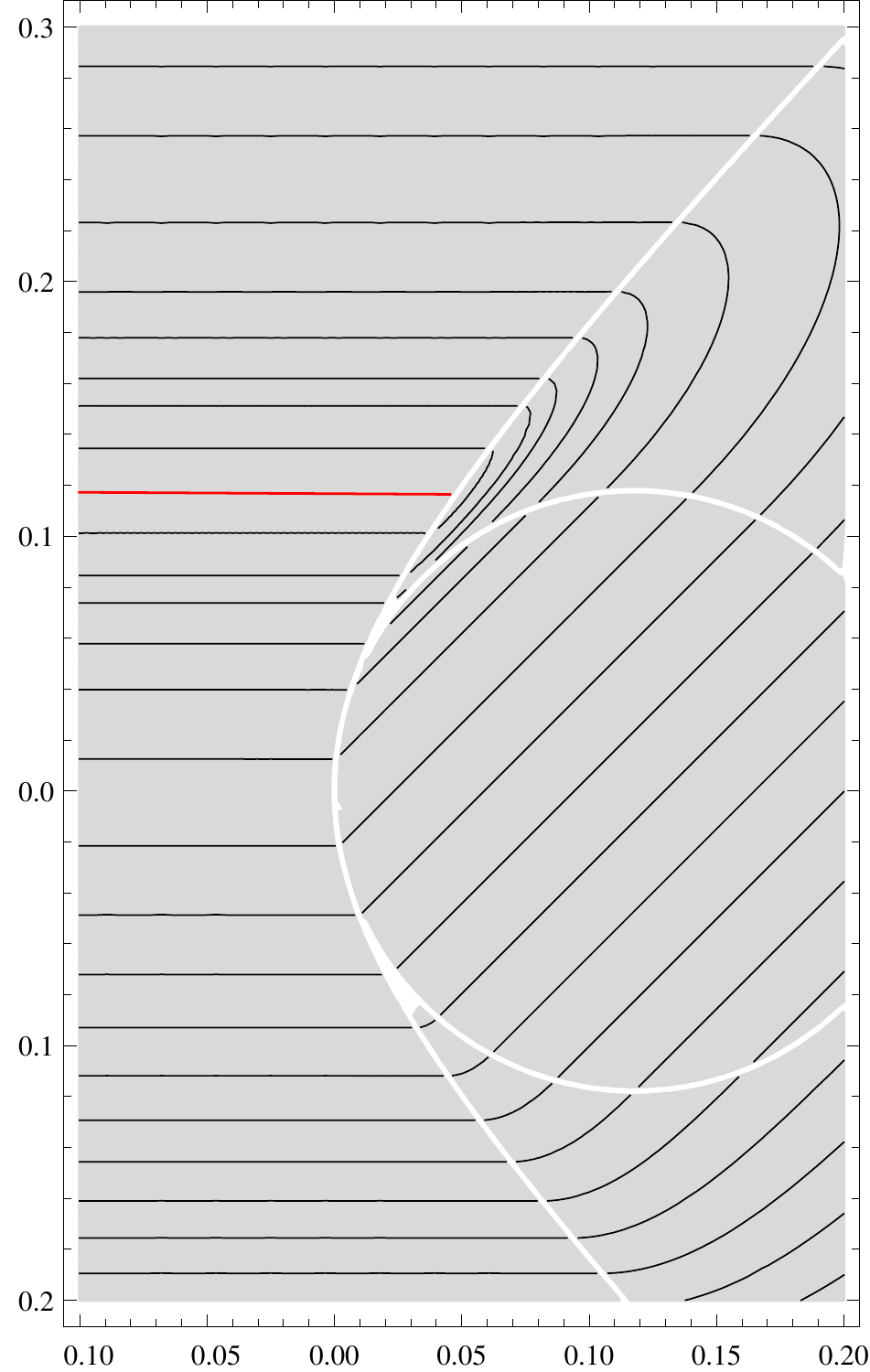}\\
\includegraphics[width=4.8cm]{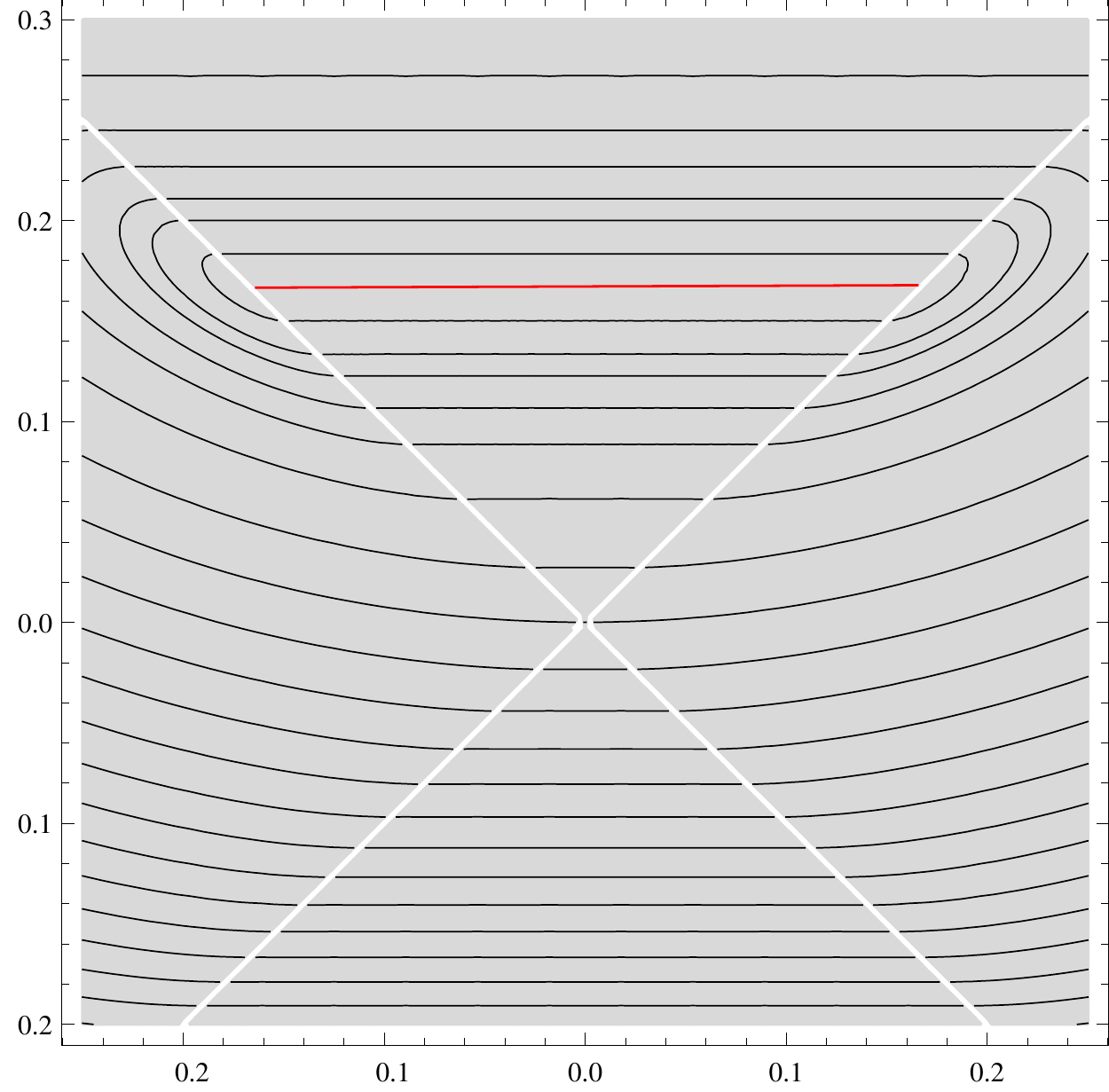}
\hspace{1cm}
\includegraphics[width=4.1cm]{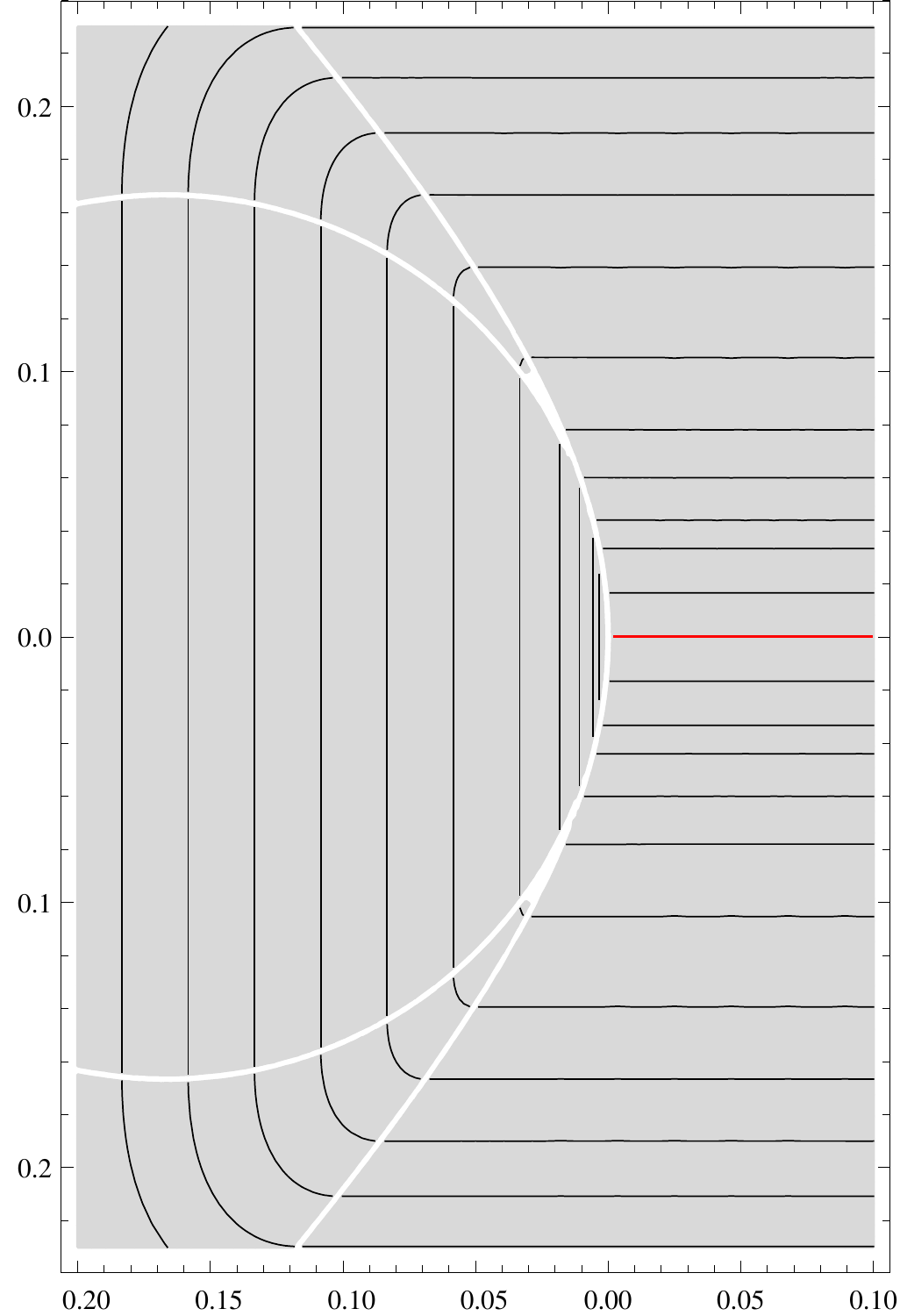}
\end{center}
\caption{Phase diagrams with level curves of 
$\overline Q_T^{\,\theta}$,
where the white lines emphasize the boundary of $\mathcal V_T$
and the red lines the set of minimisers.
The pictures from top left to bottom right correspond to the cases 
$\theta=0$, $\theta=\pi/8$, $\theta=\pi/4$, 
and $\theta=\pi/2$, respectively.} 
\end{figure}

We want to examine the minimisers of $\overline Q_T^{\,\theta}$.
Observe that in the case $\kk=0$ the above function 
$(\alpha,\beta)\mapsto\overline Q(\alpha,\beta)$
is minimised by $(0,0)$.
When instead $\kk>0$, we have that the minimisers of 
$\overline Q_T^{\,\theta}$ lie on a segment.
This is the content of the following lemma.

\begin{lemma}
\label{minimi_T}
For every $0\leq\theta<\pi$,
$\overline Q_T^{\,\theta}$ attains its minimum value
precisely on the segment
$\big[\alpha_{\theta,1}^T,\alpha_{\theta_,2}^T\big]
\times\{\beta_\theta\}$,
where
\begin{equation*}
\beta_\theta
:=
\frac{\kk\,c_1}{2\,c}\sin 2\theta,
\end{equation*}
and
\begin{equation*}
\alpha_{\theta,1}^T
:=
-\frac{\kk\,c_1}{2\,c}(1+\cos2\theta),
\qquad\qquad
\alpha_{\theta,2}^T
:=
\frac{\kk\,c_1}{2\,c}(1-\cos2\theta).
\end{equation*}
Moreover,
\begin{equation}
\label{min_Q_T}
\min_{\RR^2}\overline Q_T^{\,\theta}
\,=\,
c_1\kk^2\left(2-\frac{c_1}c\right)
+\bar e_T.
\end{equation}
\end{lemma}

Notice that the segment
$\big[\alpha_{\theta,1}^T,\alpha_{\theta_,2}^T\big]
\times\{\beta_\theta\}$
is a subset of $\mathcal U_T$
connecting the two branches of the hyperbola
$\frac{c_1}c\kk 
a_\theta\,\alpha \leq \beta^2-\alpha^2$
(a degenerate hyperbola in the case $\kk=0$ or 
$\theta\in\{\pi/4,3\pi/4\}$.

\begin{proof}
Consider the nontrivial case $\kk\neq0$.
A straightforward computation shows that the points
$(\alpha,\beta)\in\big(\alpha_{\theta,1}^T,\alpha_{\theta_,2}^T\big)
\times\{\beta_\theta\}$,
which lie in the interior of $\mathcal U_T$,
are local minimisers,
and that $\overline Q_T^{\,\theta}$ evaluated at each of these
points gives the value \eqref{min_Q_T}. 
At the same time, when 
$\mathcal D_T\neq\emptyset$, we have that 
$\nabla\overline Q_\theta(\alpha,\beta)\neq0$
for every $(\alpha,\beta)\in\mathcal D_T$, because 
$a_\theta$ and $b_\theta$ can never vanish simultaneously.
Moreover, a point $(\alpha,\beta)$ lying in the interior of
$\mathcal V_T$ is a critical point of $\overline Q_T^{\,\theta}$ 
iff
\begin{align}
c\,\alpha-c\,\frac{\beta^4}{\alpha^3}
+c_1\kk\,a_\theta\frac{\beta^2}{\alpha^2}
+c_1\kk\,a_\theta&\,=\,0,\label{prima}\\
c\,\frac{\beta^3}{\alpha^2}
+c\,\beta
-c_1\kk\,a_\theta\frac{\beta}{\alpha}
-c_1\kk\,b_\theta&\,=\,0.\label{seconda}
\end{align}
Now, observe that in the case $\beta\neq0$, multiplying
the second equation by $\beta/\alpha$ and adding the first
yields 
\begin{equation*}
c\,(\alpha^2+\beta^2)+c_1\kk
\big(a_\theta\alpha+b_\theta\beta\big)\,=\,0.
\end{equation*} 
At the same time, equation \eqref{prima} is equivalent to
\begin{equation*}
c\,(\beta^2-\alpha^2)-c_1\kk\,a_\theta\alpha\,=\,0.
\end{equation*}
Summing up the last two equations gives $\beta=c_1\kk b_\theta/(2c)$
and in turn $\alpha=\alpha_{\theta,1}^T$ or 
$\alpha=\alpha_{\theta,2}^T$, when $b_\theta\neq0$. 
In the case $b_\theta=0$, a similar argument gives that
$(-c_1\kk/c,0)$ is the solution to \eqref{prima}--\eqref{seconda}.
In any case, we have obtained that the solutions of 
\eqref{prima}--\eqref{seconda} lie in 
$\partial\mathcal U_T$. Hence, there are no critical points
of $\overline Q_T^{\,\theta}$ in the interior of $\mathcal U_T$.
Other straightforward computations show that
the values of $\overline Q_T^{\,\theta}$ on 
$\partial U_T\setminus\{(\alpha_{\theta,1}^T,\beta_\theta),
(\alpha_{\theta,2}^T,\beta_\theta)\}$ are strictly smaller than
the local minimum, therefore the local minimisers are indeed global.
This concludes the proof of the lemma.
\end{proof}

Using the above lemma we can now find the minimisers
and the minimum of our limiting functional \eqref{funct_lim_T}.
Indeed, minimising the integrand pointwise, we have that
\begin{align*}
\min_\mathcal A
\mathscr J_T^{\theta}
\,=\,
\ell\min_{\RR^{2\times2}}\overline Q_T^{\,\theta}
&\,=\,
\ell
\left[
c_1\kk^2\left(2-\frac{c_1}c\right)
+\bar e_T
\right]\\
&\,=\,
\frac{\mu\,\ell}{\pi^4}
\left[
3\left(\frac{1+2\bm}{1+\bm}\right)
+\frac{\pi^4-4\pi^2-48}8
\right]
\frac{\eta_0^2}{h_0^2},
\end{align*}
where in the second equality we have used \eqref{min_Q_T}
and in the last one the constants $c_1$, $c$, $\kk$, and $\bar e_T$
have been substituted with their expressions given 
in terms of the parameters of the $3$D model.
The set 
$(\alpha,\beta)\in\big[\alpha_{\theta,1}^T,\alpha_{\theta_,2}^T\big]
\times\{\beta_\theta\}$
of the minimisers of $\overline Q_T^{\,\theta}$ is given by
\begin{equation*}
\left[\,
\frac{3\,\eta_0}{\pi^2(1+\bm)\,h_0}(-a_\theta-1)\,,\,
\frac{3\,\eta_0}{\pi^2(1+\bm)\,h_0}(1-a_\theta)
\,\right]
\times
\left\{\frac{3\,\eta_0}{\pi^2(1+\bm)\,h_0}b_\theta\right\}.
\end{equation*}
Hence, a minimiser of $\mathscr J_T^{\theta}$ is any
$(d_1,d_2,d_3)\in\mathcal A$ such that
$d_1^\prime\cdot d_3$ and $d_2^\prime\cdot d_3$ are constant 
and satisfy
\begin{equation*}
d_1^\prime\cdot d_3
\in
\left[\,
\frac{3\,\eta_0}{\pi^2(1+\bm)\,h_0}(-a_\theta-1)\,,\,
\frac{3\,\eta_0}{\pi^2(1+\bm)\,h_0}(1-a_\theta)
\,\right],
\quad
d_2^\prime\cdot d_3
\,=\,
\frac{3\,\eta_0}{\pi^2(1+\bm)\,h_0}b_\theta.
\end{equation*}
Notice that when $\alpha$ and $\beta$ are real constants,
the problem
\begin{equation}
\label{ODINA}
(d_1,d_2,d_3)\in\mathcal A,
\qquad\qquad
d_1^\prime\cdot d_3=\alpha,
\qquad\qquad
d_2^\prime\cdot d_3=\beta,
\end{equation}
has always a solution.
Indeed, identifying $(d_1,d_2,d_2)$ with the rotation matrix
$Q=(d_1|d_2|d_2)$, it is standard to see that finding a 
solution of \eqref{ODINA} consists in solving
\[
Q^\prime(s) 
\,=\, 
Q(s)  \left(\begin{array}{ccc} 0 & 0 & -\alpha \\  0 & 0 & -\beta \\ \alpha & \beta & 0\end{array}\right),
\]
for some fixed $Q(0)\in{\rm SO}(3)$.
Also, once $s\mapsto d_1(s)$ is given,
the mid-line curve is given by
\[
r(s) = r(0) + \int_0^s d_1(\sigma)d\sigma
\]
is fixed up to a translation.

For the convenience of the reader, let us make explicit the 
condition for which $\mathscr J_T^\theta$ is minimised
in two cases:
\begin{align*}
&\theta=0:
\qquad\,
(d_1^\prime\cdot d_3,d_2^\prime\cdot d_3)
\in
\left[\,-
\frac{6\,\eta_0}{\pi^2(1+\bm)\,h_0}\,,\,0
\,\right]
\times
\left\{0\right\},\\
&\theta=\pi/4:
\quad
(d_1^\prime\cdot d_3,d_2^\prime\cdot d_3)
\in
\left[\,-
\frac{3\,\eta_0}{\pi^2(1+\bm)\,h_0}\,,\,
\frac{3\,\eta_0}{\pi^2(1+\bm)\,h_0}
\,\right]
\times
\left\{\frac{3\,\eta_0}{\pi^2(1+\bm)\,h_0}\right\}.
\end{align*}

\begin{figure}[htbp]
\label{elicoidi_T_theta_0}
\begin{center}
\includegraphics[width=3cm]
{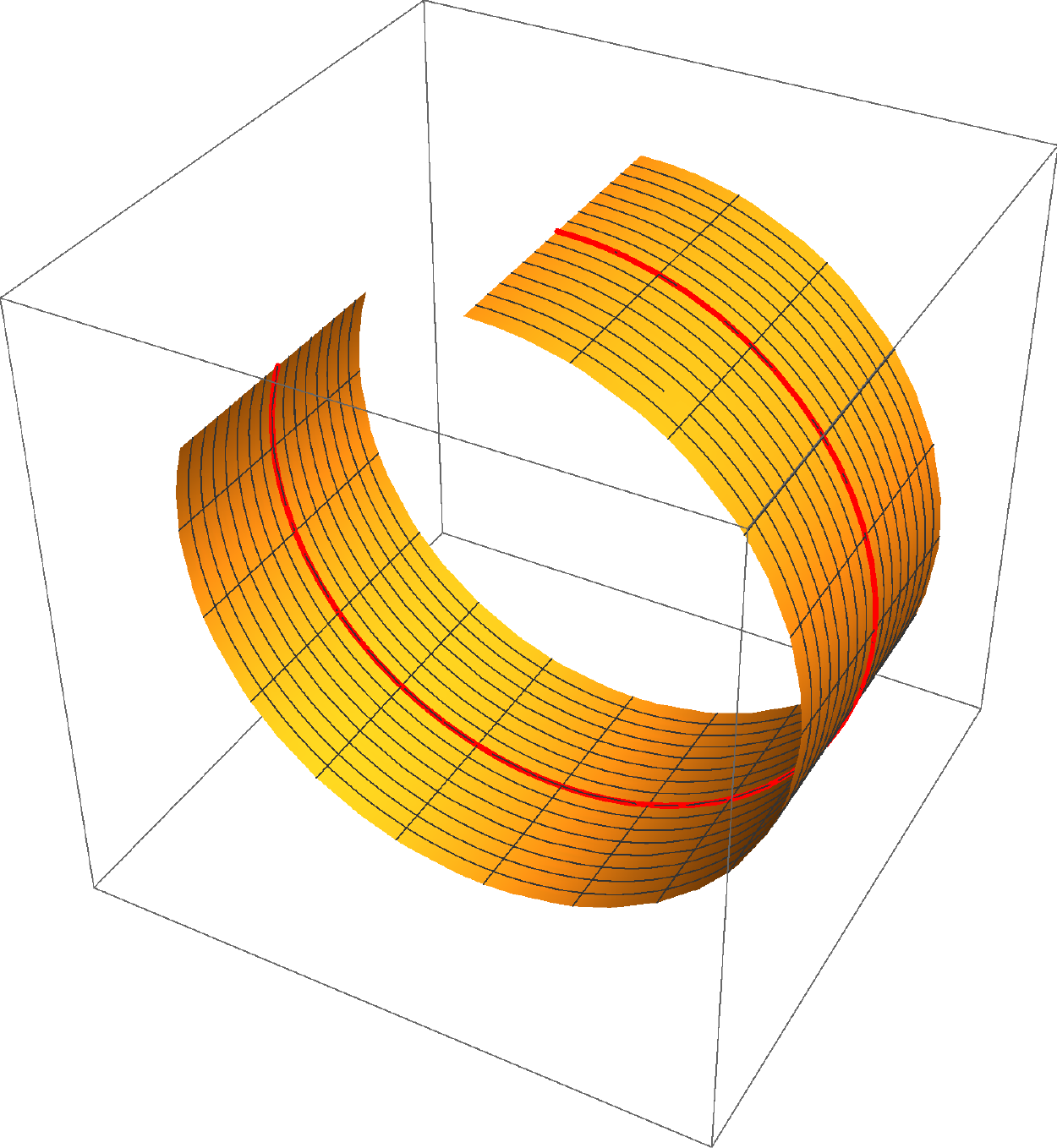}
\hspace{1.1cm}
\includegraphics[width=3cm]
{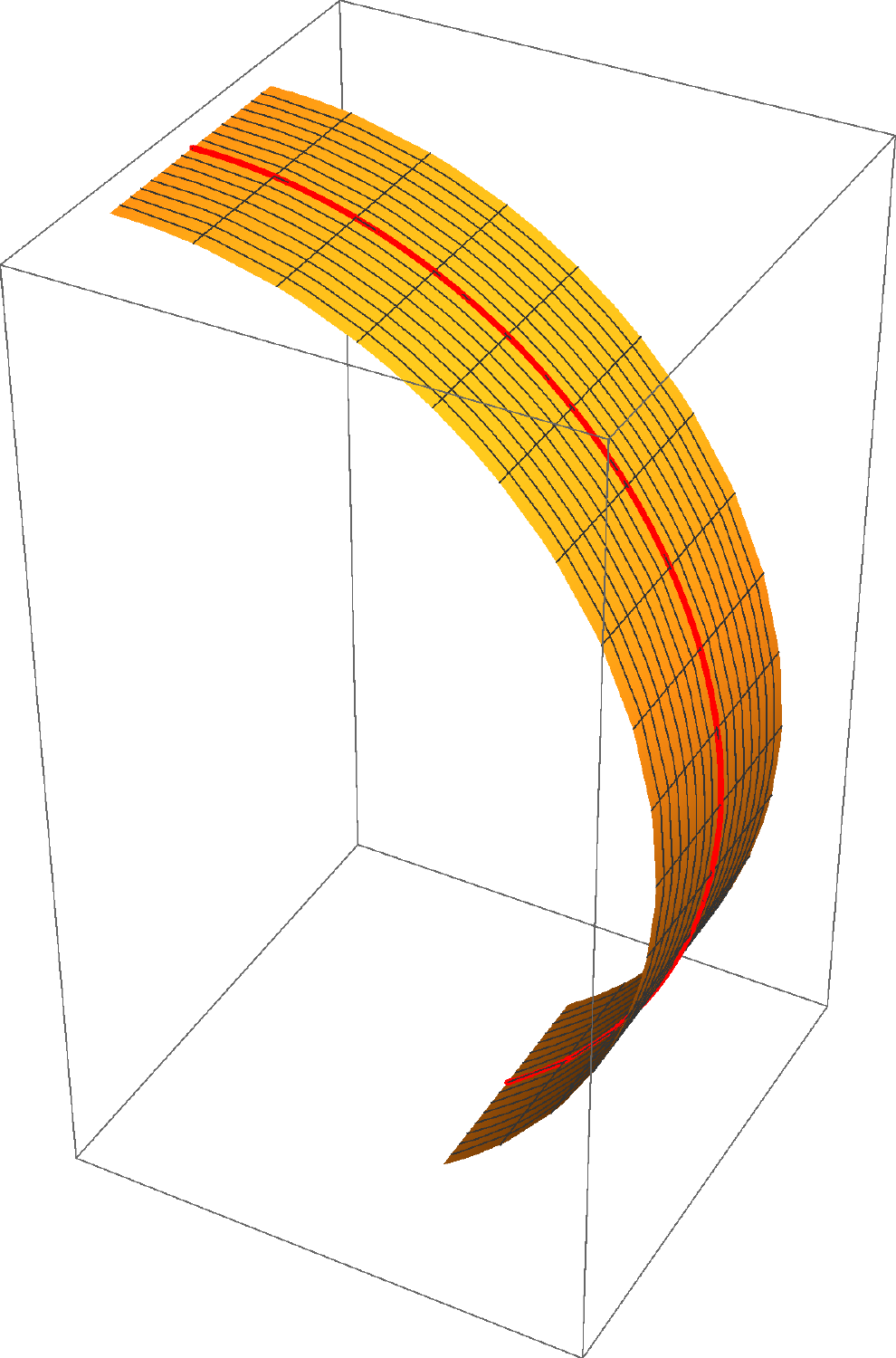}
\hspace{1.1cm}
\includegraphics[width=4cm]
{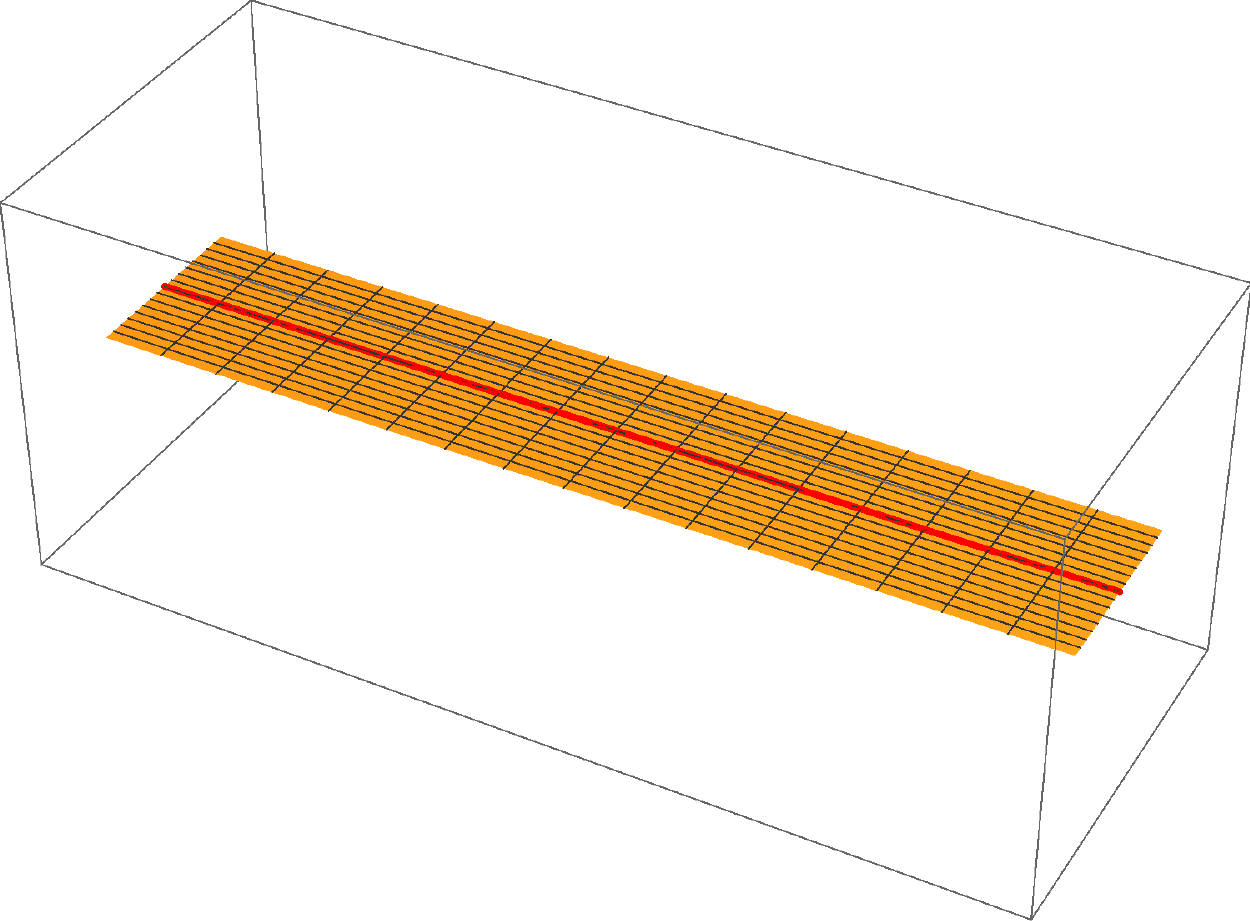}
\end{center}
\caption{Minimal energy configurations for the cases $\theta=0$
and $\theta=\pi/2$,
with a decreasing value of $|\alpha|=|d_1^\prime\cdot d_3|$
from left to right,
and with $\beta=d_2^\prime\cdot d_3=0$.
The mid-line is shown in red. 
}
\end{figure} 

\begin{figure}[htbp]
\label{elicoidi_T_theta_pi4}
\begin{center}
\includegraphics[width=6cm]
{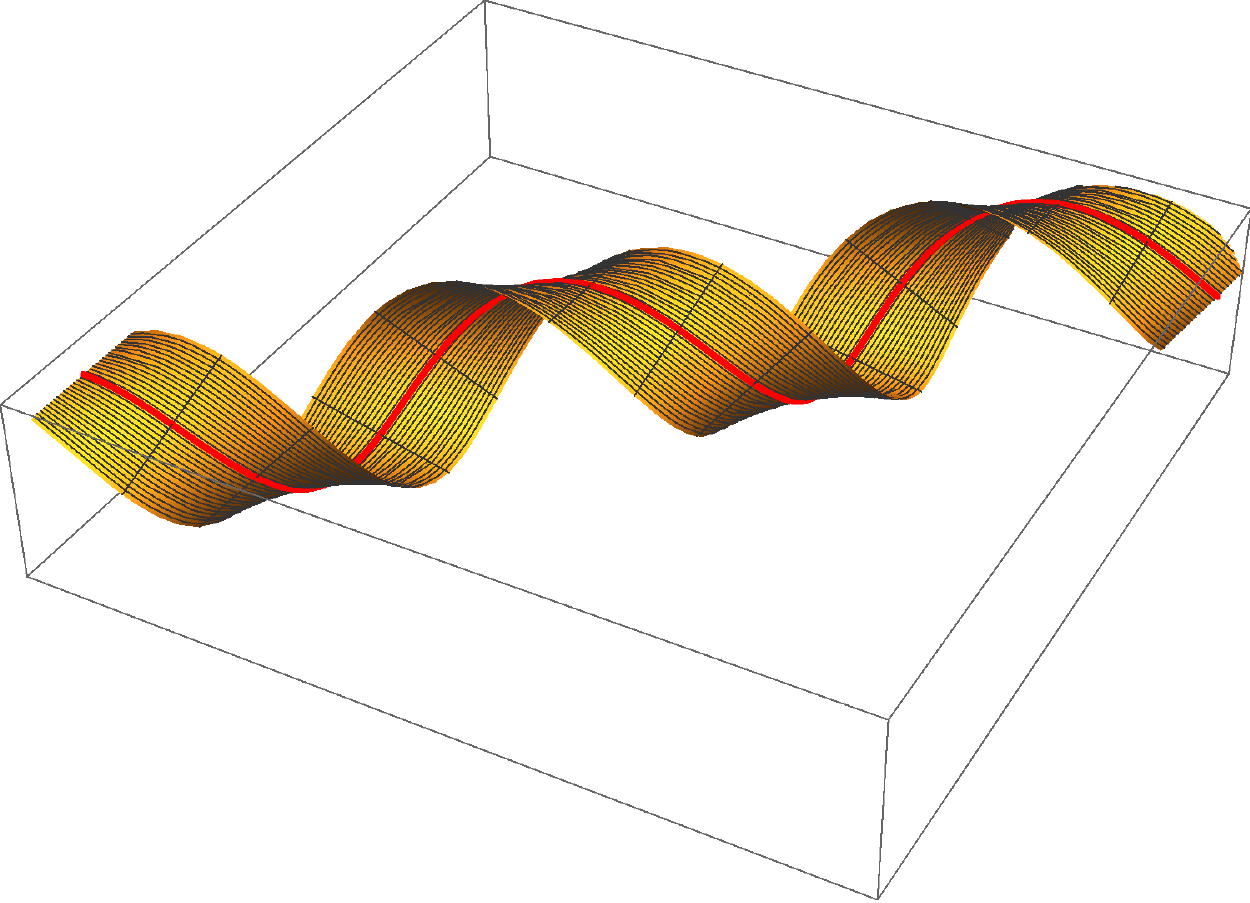}
\hspace{.6cm}
\includegraphics[width=5.5cm]
{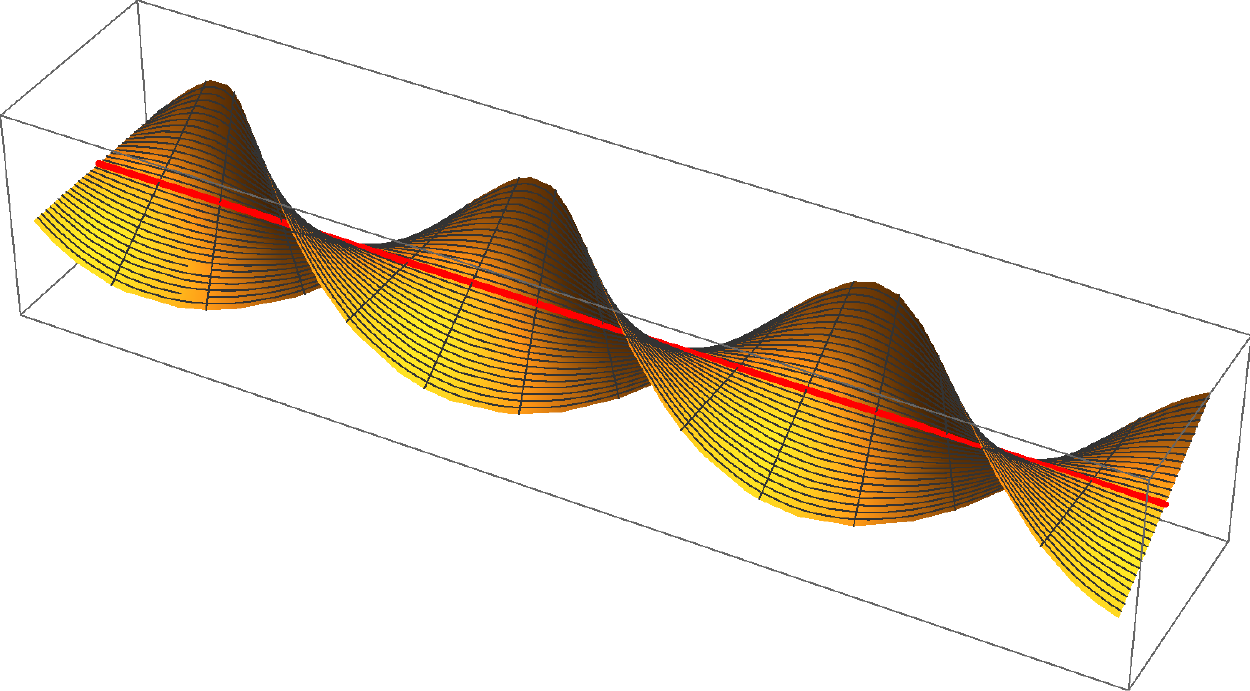}
\end{center}
\caption{
A selection of minimal energy configurations 
for the cases $\theta\in(0,\pi/2)\cup(\pi/2,\pi)$.
Both of them correspond to some
positive and constant $\beta=d_2^\prime\cdot d_3$
and some constant $\alpha=d_1^\prime\cdot d_3$.
In particular, the second configuration 
corresponds to $\alpha=0$ and its mid-line 
(shown in red) is a straight line.
In the other configuration the mid-line is a helix.
}
\end{figure} 

We note that the minimisers predicted by our limiting model
are in agreement with  
\cite[Figure 3]{Studart}, where minimum-energy
configurations of self-shaping synthetic systems with oriented reinforcement are shown.

\section{The splay-bend case}
\label{splay-bend}

In the splay-bend case, 
the $\theta$-dependent target curvature tensor
$\bar A^\theta$ defined in \eqref{A_theta} is 
\begin{equation*}
\bar A_S^\theta 
=
\kk \left( \begin{array}{cc} -\cos^2\theta & \sin\theta\cos\theta \\ \sin\theta\cos\theta & -\sin^2\theta \end{array}\right) 
\end{equation*}
and $\det \bar A_S^\theta =\det \bar A_S= 0$.
The functional defined in \eqref{eq:bendingenergy1} becomes
\begin{equation*}
\mathscr E_{\ep,S}^\theta(v)
=
\frac1{\ep}\int_{S_\ep}
\Big\{c |A_v(z)|^2 + L_S^\theta(A_v(z))\Big\}{\rm d}z,
\end{equation*}
with
\[
L_S^\theta(A_v) 
\,:=\,  
- 2c_1 A_v\cdot\bar{A}^{\theta,SB} + 2c_2\kk\tr A_v
+ c \kk^2 + \bar e_S.
\]
Also, the rescaled energy \eqref{our_fun} is now
\begin{equation*}
\mathscr J_{\ep,S}^{\theta}(y) 
:= 
\int_S \Big\{c |A_{y,\ep}(x)|^2 + L_S^\theta(A_{y,\ep}(x))\Big\}\dx,
\end{equation*}
for every $y\in\W^{2,2}_{{\rm iso},\ep}(S,\RR^3)$,
and
$\mathscr J_S^{\theta}:\mathcal{A}\to \RR$ is defined as
\begin{equation}
\label{funct_lim_S}
\mathscr J_S^{\theta}(d_1,d_2,d_3)
\,:=\,
\int_I 
\overline Q_S^{\,\theta}
(d_1^\prime\cdot d_3,d_2^\prime\cdot d_3)\dx_1,
\end{equation}
where $\overline Q_S^{\,\theta}$
is given by \eqref{Qasmin}, 
with $L_S^\theta$ in place of $L^\theta$.
The counterpart of Corollary~\ref{cor:twist}
holds for the splay-bend case: it is sufficient to replace
$\mathscr J_{\ep,T}^{\theta}$ and $\mathscr J_T^{\theta}$
by 
$\mathscr J_{\ep,S}^{\theta}$ and $\mathscr J_S^{\theta}$,
respectively, in the statement
of Corollary \ref{cor:twist}.

Note that $\bar A_S^\theta$ can be alternatively written as
\[
\bar A_S^\theta = \frac12 \left(\bar A_T^\theta - \mathbb{I}\right)
\]
and in turn
\[
L_S^\theta(A) 
= -c_1 A\cdot \bar A_T^\theta + k(c+c_2)\tr A + ck^2 + \bar e_S.
\]
Hence, setting $d_\theta = c_1a_\theta -c -c_2$,
we have that
\[
\overline Q_S^{\,\theta}(\alpha,\beta) 
=\min_{\gamma\in\RR} f(\gamma),
\]
where
\begin{align*}
f(\gamma) 
&:= 
c(\alpha^2 + 2\beta^2 +\gamma^2) + 2c|\alpha\gamma - \beta^2| 
-k d_\theta(\alpha+\gamma) + 2kc_1(a_\theta \alpha - b_\theta\beta) 
 + c\kk^2 + \bar e_S,
\end{align*}
recalling that
$a_\theta:=\cos2\theta$ and $b_\theta:=\sin2\theta$.
This will be useful in the proof of the following
proposition.

\begin{proposition}
\label{prop:QS}
$\overline Q_S^{\,\theta}$ is a continuous function given by
\[
\overline Q_S^{\,\theta}(\alpha,\beta)= 
\left\{\begin{array}{lc}
2c_1\kk\big(a_\theta\alpha - b_\theta\beta\big)
+c\kk^2\Big(1-\frac{d_\theta^2}{4c^2}\Big) + \bar e_S, 
& 
\mbox{in } \mathcal D_S\\

 4c\beta^2 - 2kd_\theta\alpha
 +2c_1\kk\big(a_\theta\alpha - b_\theta\beta\big)
+c\kk^2\Big(1-\frac{d_\theta^2}{4c^2}\Big) + \bar e_S,  
&
\mbox{in } \mathcal U_S\\

c\frac{(\alpha^2+\beta^2)^2}{\alpha^2} - 
kd_\theta\frac{\alpha^2+\beta^2}{\alpha}
+2c_1\kk\big(a_\theta\alpha - b_\theta\beta\big)
+c\kk^2+\bar e_S,\quad  
&
\mbox{in } \mathcal V_S.
 \end{array}\right.
\]
where $d_\theta := c_1a_\theta - c - c_2$, and
\begin{align*}
\mathcal D_S
&:=\left\{(\alpha,\beta)\in\RR^2:\frac{\kk}{2c}
d_\theta\,\alpha > \beta^2+\alpha^2\right\},\\
\mathcal U_S
&:=\left\{(\alpha,\beta)\in\RR^2:\frac{\kk}{2c} 
d_\theta\,\alpha \leq \beta^2-\alpha^2\right\},\\
\mathcal V_S
&:=\RR^2\setminus(\mathcal D_S\cup\mathcal U_S).
\end{align*}
\end{proposition}

\begin{figure}[htbp]
\label{figure_T}
\begin{center}
\includegraphics[width=4.4cm]{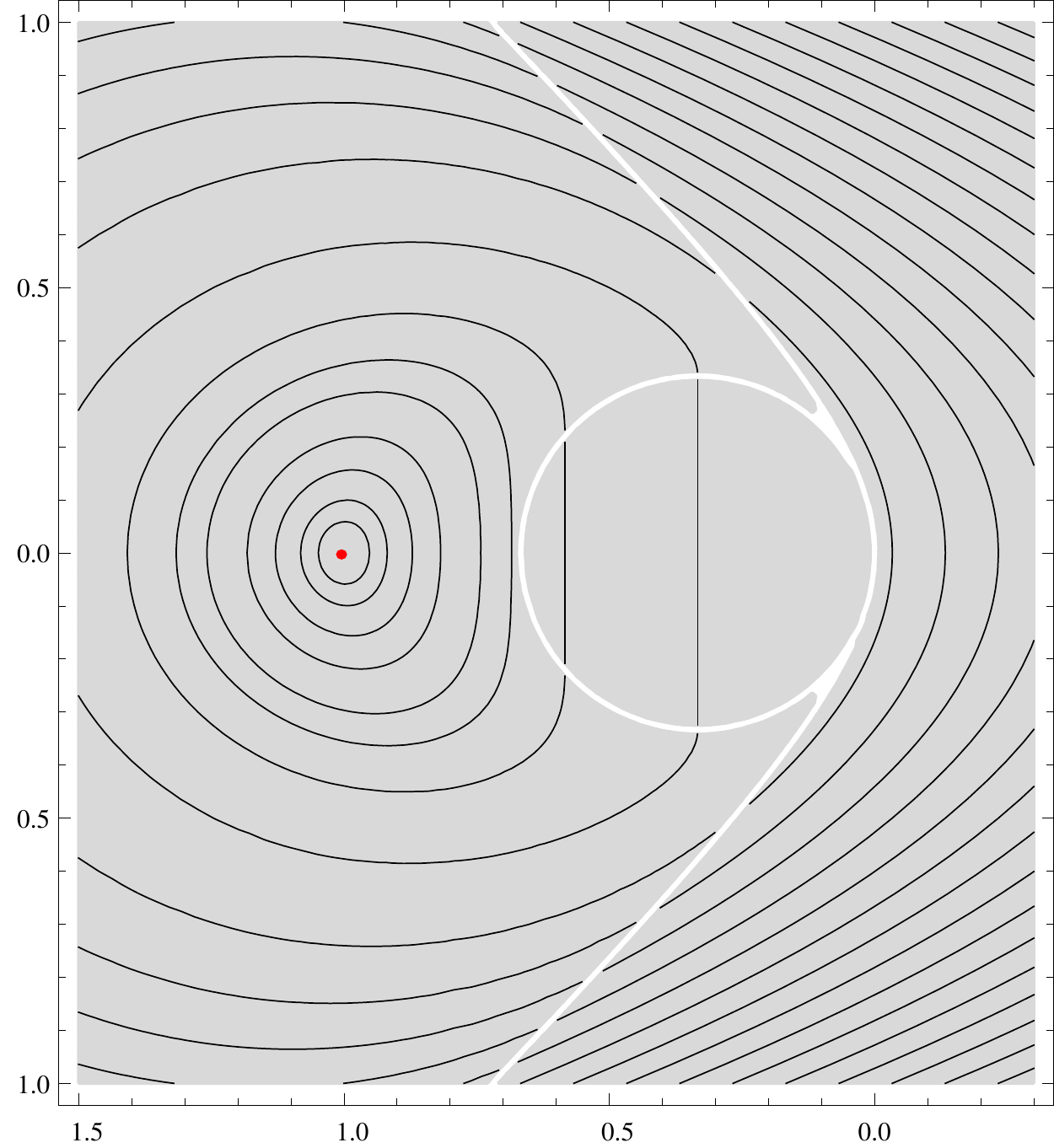}
\hspace{1cm}
\includegraphics[width=4.4cm]{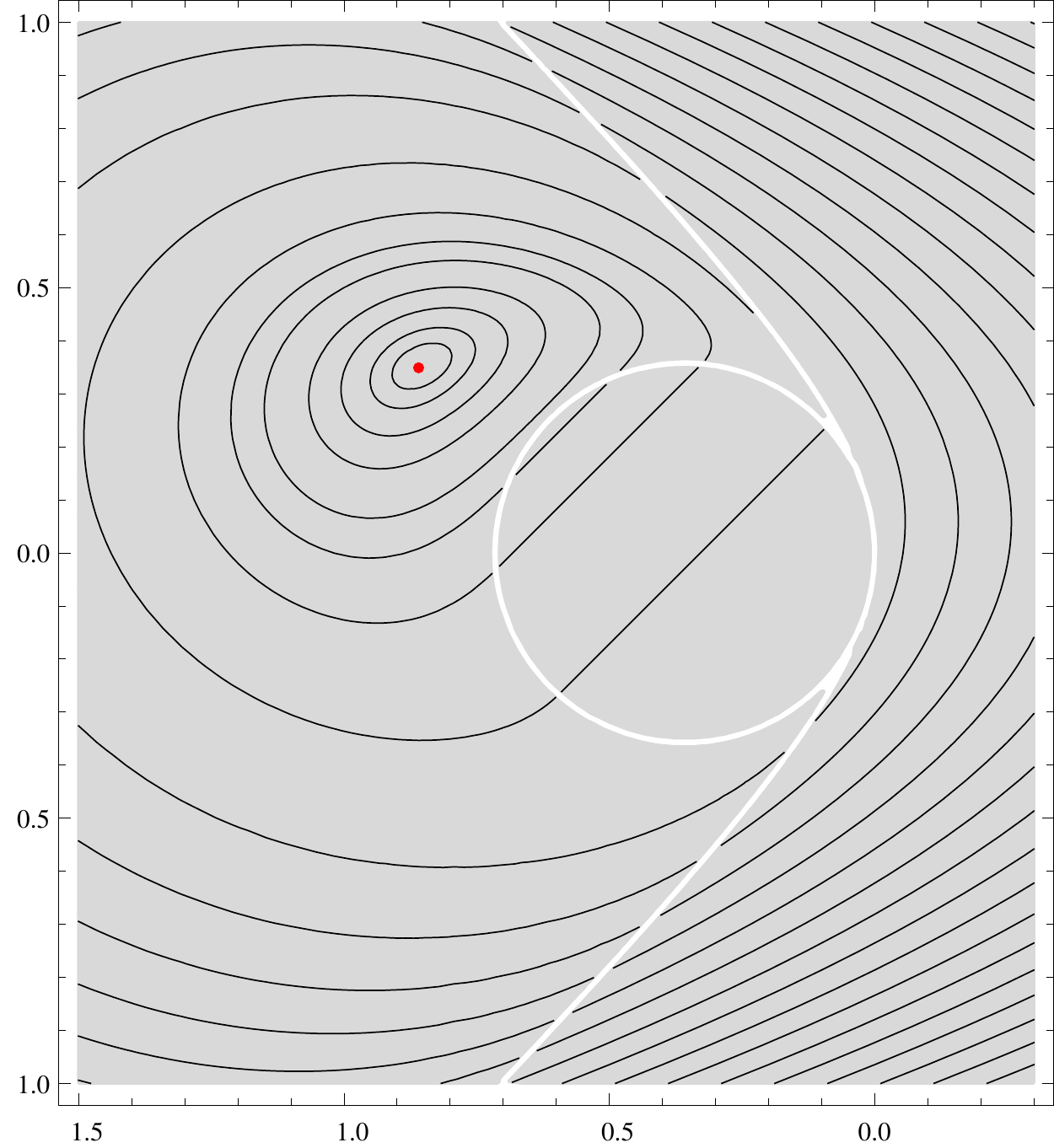}\\
\includegraphics[width=4.4cm]{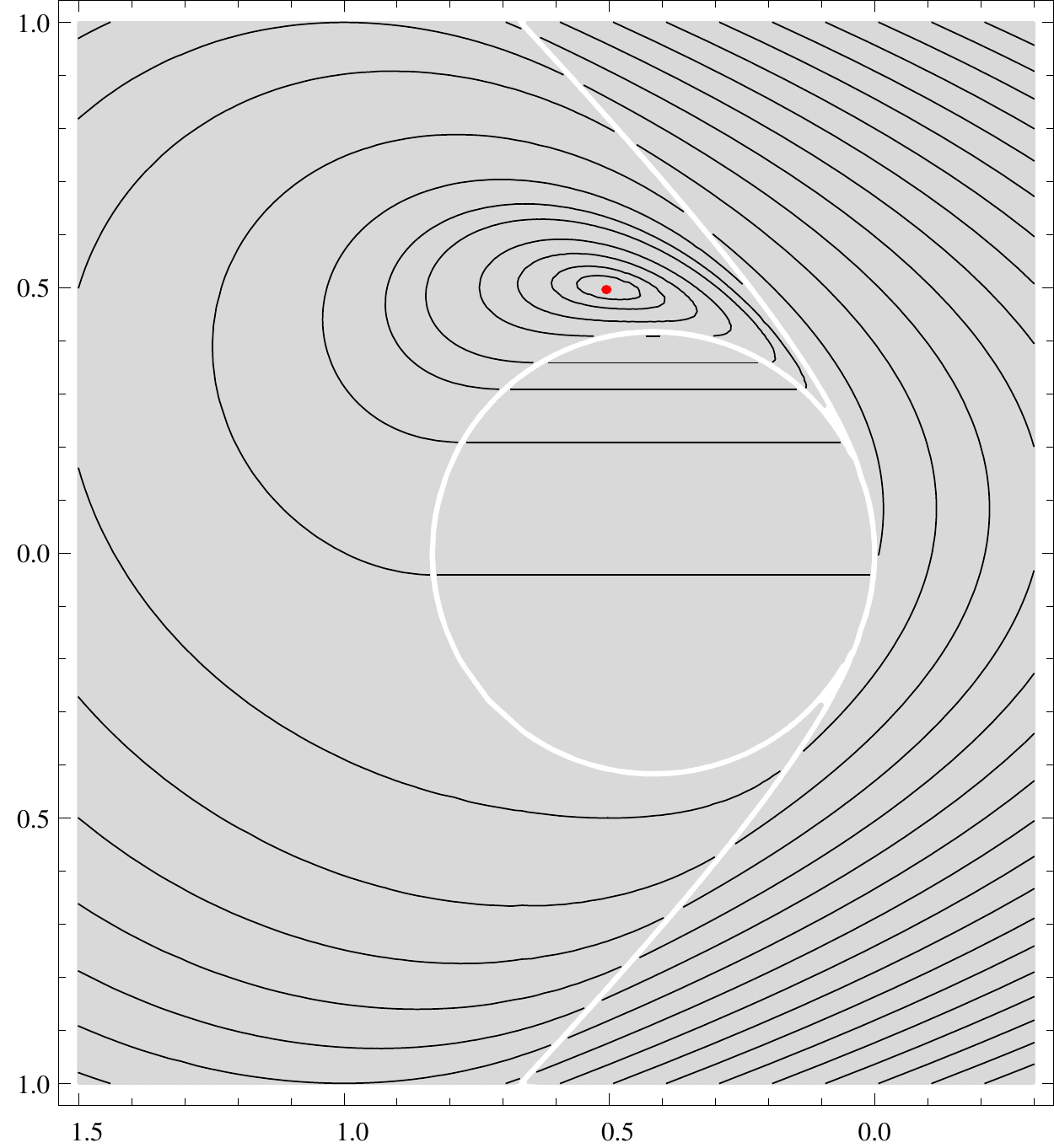}
\hspace{1cm}
\includegraphics[width=5.2cm]{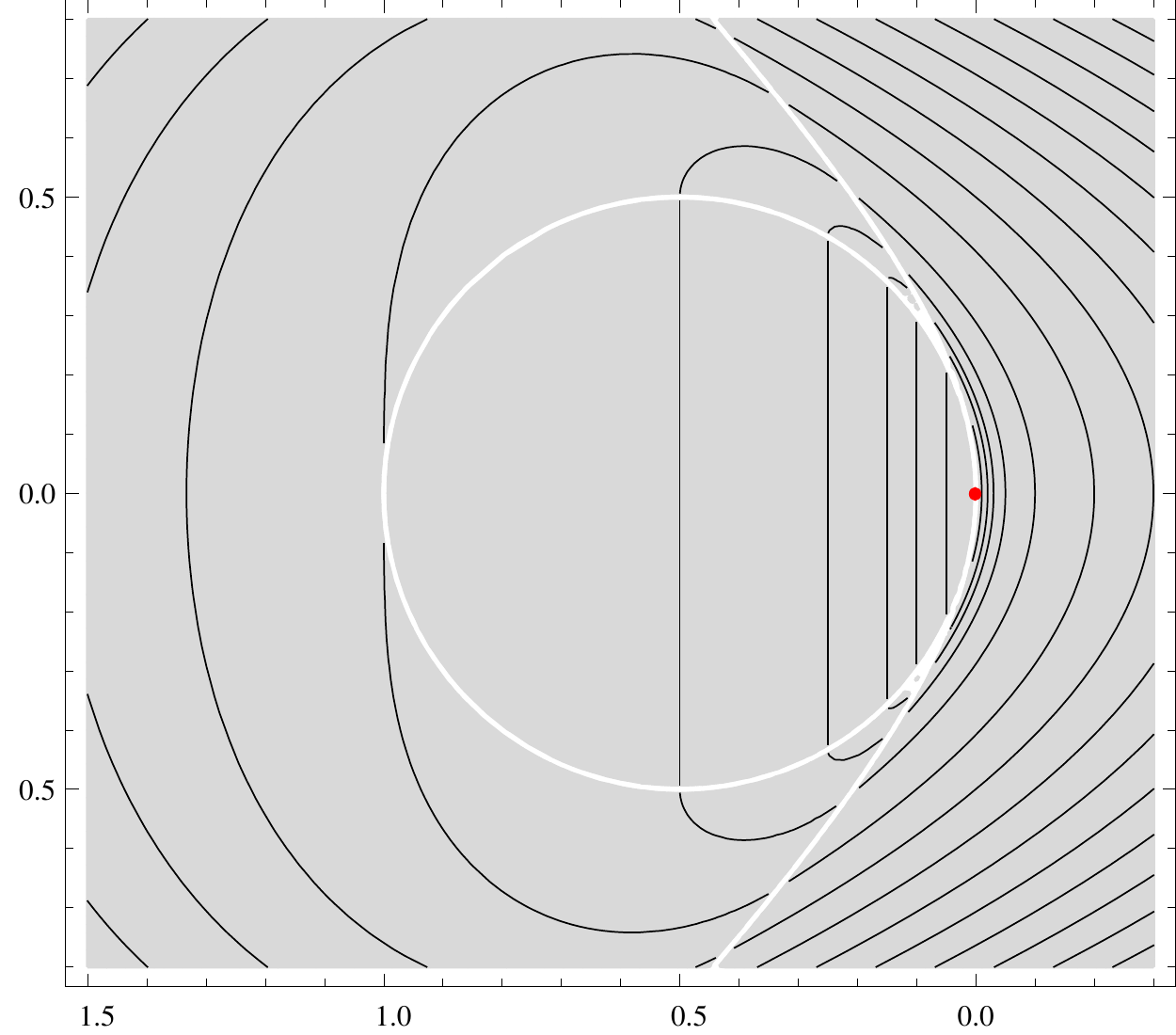}
\end{center}
\caption{Phase diagrams with level curves of 
$\overline Q_S^{\,\theta}$,
where the white lines emphasize the boundary of $\mathcal V_S$
and the red dots are the minimisers.
The pictures from top left to bottom right
correspond to the cases 
$\theta=0$, $\theta=\pi/8$, $\theta=\pi/4$, and $\theta=\pi/2$,
respectively.} 
\end{figure}

\begin{remark}
\label{remarkQ_S}
Setting 
\[
\overline Q_{S,1}^{\,\theta}
\,:=\,
2c_1\kk\big(a_\theta\alpha - b_\theta\beta\big)
+c\kk^2\Big(1-\frac{d_\theta^2}{4c^2}\Big) + \bar e_S,
\]
we have that $\overline Q_S^{\,\theta}=\overline Q_{S,1}^{\,\theta}$
in $\mathcal D_S$, that
\[
\overline Q_S^{\,\theta}
\,=\,
\overline Q_{S,2}^{\,\theta}(\alpha,\beta)
\,:=\,
4c\beta^2 - 2kd_\theta\alpha
+\overline Q_{S,1}^{\,\theta}(\alpha,\beta) 
\qquad
\mbox{in}\quad\mathcal U_S,
\]
and that
\[
\overline Q_S^{\,\theta}
\,=\,
c\Big(\frac{\alpha^2+\beta^2}{\alpha} - 
\frac{kd_\theta}{2c}\Big)^2 
\!\!+ \overline Q_{S,1}^{\,\theta}(\alpha,\beta) 
\,=\,  c\Big(\frac{\beta^2 - \alpha^2}{\alpha} - 
\frac{kd_\theta}{2c}\Big)^2 + 
\overline Q_{S,2}^{\,\theta}(\alpha,\beta)
\quad
\mbox{in}\quad\mathcal V_S.
\]
Since the squares in this expression vanishes on $\partial\mathcal D_S$
and $\partial\mathcal U_S$, respectively, this shows
in particular that $\overline Q_S^{\,\theta}$ is continuous.
\end{remark}

\begin{proof}[Proof of Proposition \ref{prop:QS}]
The proof is almost identical to that for $\overline Q_T^{\,\theta}$. 
For $\alpha = 0$, $f(\gamma)$ reduces to the
differentiable function
\[
f(\gamma) = 4c\beta^2 +c\gamma^2 
-2k d_\theta\gamma - 2c_1 \kk b_\theta\beta 
+ c\kk^2 + \bar e_S,
\]
which is minimised at $\gamma = \frac{\kk}{2c}d_\theta$ and hence, for all $\beta\in\RR$,
\[
\overline Q_S^{\,\theta}(0,\beta) = 
 4c\beta^2 - 2c_1\kk b_\theta\beta
+c\kk^2 -\frac{\kk^2 d_\theta^2}{4c} +\bar e_S.
\]
If $\alpha\neq 0$ and $\gamma = \beta^2/\alpha$,
\[
f(\beta^2/\alpha) = c\frac{(\alpha^2+\beta^2)^2}{\alpha^2} - 
kd_\theta\frac{\alpha^2+\beta^2}{\alpha}
+2c_1\kk\big(a_\theta\alpha - b_\theta\beta\big)
+c\kk^2+\bar e_S
\]
whereas, for $\gamma\neq\beta^2/\alpha$, $f$ is a differentiable function with
\[
f^\prime(\gamma) = 2c\gamma + 2c\alpha\sgn(\alpha\gamma-\beta^2) 
- kd_\theta.
\]
The critical points are then given by
\[
\gamma_1 = \frac{\kk}{2c}d_\theta - \alpha,\mbox{ in the regime }\frac{\kk}{2c}d_\theta\alpha > \alpha^2+\beta^2
\] 
and by
\[
\gamma_2 = \frac{\kk}{2c}d_\theta + \alpha,\mbox{ in the regime }\frac{\kk}{2c}d_\theta\alpha < \beta^2 - \alpha^2.
\]
The respective values of $f$ are
\[
f(\gamma_1) = 2c_1\kk(a_\theta\alpha - b_\theta\beta)
 +c\kk^2 -\frac{\kk^2 d_\theta^2}{4c} +\bar e_S
\]
and
\[
f(\gamma_2) =  4c\beta^2 - 2\kk d_\theta\alpha + 2c_1\kk\left(a_\theta\alpha - b_\theta\beta\right) 
 +c\kk^2 -\frac{\kk^2 d_\theta^2}{4c} +\bar e_S.
\]
It is easy to compute that
\begin{align*}
f(\gamma_1) - f(\beta^2/\alpha) & = -\frac{1}{4c}\kk^2 d_\theta^2
 - c\frac{(\alpha^2+\beta^2)^2}{\alpha^2} 
 +\kk d_\theta\frac{\alpha^2 + \beta^2}{\alpha}\\
& = -c\left[\frac{\alpha^2 + \beta^2}{\alpha} - \frac{\kk}{2c}d_\theta\right]^2\leq 0
\end{align*}
with equality if and only if $\frac{\kk}{2c}d_\theta\alpha = \alpha^2 + \beta^2$, i.e. in the regime $\mathcal D_S$,
\[
\overline Q_S^{\,\theta}(\alpha,\beta) = f(\gamma_1).
\]
Similarly, we find that
\begin{align*}
f(\gamma_2) - f(\beta^2/\alpha) & =  
-\frac{1}{4c}\kk^2 d_\theta^2 + 4c\beta^2
-2\kk d_\theta\alpha - c\frac{(\alpha^2+\beta^2)^2}{\alpha^2} 
+\kk d_\theta \frac{\alpha^2+\beta^2}{\alpha} \\
& = -\frac{1}{4c}\kk^2 d_\theta^2 
- c\frac{(\alpha^2 - \beta^2)^2}{\alpha^2}
+ \kk d_\theta \frac{\beta^2 - \alpha^2}{\alpha}\\
& = -c\left[\frac{\beta^2 - \alpha^2}{\alpha} - \frac{\kk}{2c}d_\theta\right]^2\leq 0
\end{align*}
with equality if and only if $\frac{\kk}{2c}d_\theta\alpha = \beta^2 - \alpha^2$, i.e. in the regime $\frac{\kk}{2c}d_\theta < \beta^2 - \alpha^2$,
\[
\overline Q_S^{\,\theta}(\alpha,\beta) = f(\gamma_2).
\]
At the same time, the calculations above also establish the Remark following Proposition \ref{prop:QS}.
To conclude the proof, a straightforward computation shows that 
$f^\prime(\gamma)<0$ if $\gamma<\beta^2/\alpha$ and $f^\prime(\gamma)>0$ if $\gamma>\beta^2/\alpha$. Hence, in the regime $\alpha\neq 0$ and
\[
\beta^2 - \alpha^2 \leq \frac{\kk}{2c}d_\theta\alpha \leq \alpha^2 + \beta^2
\]
the minimum value of $f$ is achieved at $\gamma=\beta^2/\alpha$ and
\[
\overline Q_S^{\,\theta}(\alpha,\beta) = f(\beta^2/\alpha).
\] 
\end{proof}

We now focus on the minimisers of $\overline Q_S^{\,\theta}$.
As for the twist case, when $\kk=0$ and up to
additive and multiplicative constants, 
the function $\overline Q_S^{\,\theta}$ reduces to
the function defined in \eqref{bar_Q}, 
which is minimised at $(0,0)$. 
When instead $\kk>0$,  differently from the twist case we have that
for every $\theta$ the minimiser of $\overline Q_S^{\,\theta}$
is a ($\theta$-dependent) single point,
in view of the following lemma.

\begin{lemma}
\label{minimi_S}
For every $0\leq\theta<\pi$,
$\overline Q_S^{\,\theta}$ is minimised precisely at
$(\alpha_{\theta}^S,\beta_\theta^S)$,
where
\begin{equation*}
\alpha_{\theta}^S
\,:=\,
-\frac k2\big(1+\cos2\theta\big),
\qquad\qquad
\beta_\theta^S
\,:=\,
\frac k2\sin2\theta.
\end{equation*}
Moreover,
\begin{equation}
\label{min_Q_S}
\min_{\RR^2}\overline Q_S^{\,\theta}
\,=\,
\bar e_S.
\end{equation}
\end{lemma}

\begin{proof}
Consider the nontrivial case $\kk\neq0$.
A straightforward computation shows that there are no
critical points of $\overline Q^{\,\theta}_S$
in the interior of $\mathcal D_S$ or $\mathcal U_S$,
for any value of $\theta$. Next, we look for critical points 
in the interior of
$\mathcal V_S$. 
Differentiating the first of the two expressions for 
$\overline Q_S^{\,\theta}$ in $\mathcal V_S$ 
given in Remark \ref{remarkQ_S} 
and setting $\nabla\overline Q_S^{\,\theta}=0$ yields
\begin{align}\label{eq:proof_Q_S_1}
2c\Big(\frac{\alpha^2 + \beta^2}{\alpha} - \frac{k d_\theta}{2c}\Big)\frac{\beta^2 - \alpha^2}{\alpha^2} 
&\, =\, 
2c_1k a_\theta \\
\,&\,\nonumber\\
2c\Big(\frac{\alpha^2 + \beta^2}{\alpha} - \frac{k d_\theta}{2c}\Big)\frac{\beta}{\alpha} 
& \,=\,
 c_1k b_\theta.
\label{eq:proof_Q_S_2}
\end{align}
We first examine the cases $\theta = 0$ and $\theta = \pi/2$.
In these cases we have $b_\theta = 0$ and $|a_\theta|=1$,
so that from \eqref{eq:proof_Q_S_2} we get $\beta=0$,
since $(\alpha^2 + \beta^2)/\alpha\neq(k d_\theta/(2c)$
in the interior of $\mathcal V_S$.
Setting $\beta=0$ in \eqref{eq:proof_Q_S_1} we then get 
$
\alpha = -\frac{k}{2c}(c_1a_\theta + c + c_2).
$ 
Now, an easy calculation shows that 
in the case $\theta = 0$ the point 
$(\alpha_0^S,\beta_0^S)$ lies in the interior of $\mathcal V_S$
and it is the global minimiser of $\overline{Q}_S^{\,\theta}$.
On the other hand, in the case $\theta=\pi/2$
another easy computation show that the
the point $(-kc_2/c , 0)$ lies in $\mathcal D_S$ and it 
is thus not a critical point of our function in the interior of 
$\mathcal V_S$. 
Then, comparing the values of $\overline Q_S^{\,\pi/2}$ 
on the boundaries of $\mathcal D_S$ and $\mathcal U_S$, 
we find that the minimum is achieved at 
$(\alpha_{\pi/2}^S,\beta_{\pi/2}^S) = (0,0)$.
Moreover, it can be readily checked that  
\eqref{min_Q_S} holds for $\theta=0$ and $\theta=\pi/2$.

Consider now an arbitrary
$\theta\in(0,\pi)\setminus\{\pi/2\}$
and note that in this case $b_\theta\neq 0$ and $|a_\theta| < 1$.
As before, we search for critical points of 
$\overline Q_S^{\,\theta}$ in the interior
of $\mathcal V_S$. 
From \eqref{eq:proof_Q_S_2} we get in particular $\beta\neq 0$.
Therefore, we may divide \eqref{eq:proof_Q_S_1} by 
\eqref{eq:proof_Q_S_2} getting
\begin{equation}
\label{eq:proof_Q_S_3}
\frac{\beta^2 - \alpha^2}{\alpha} = 2\frac{a_\theta}{b_\theta}\beta,
\end{equation}
and in turn that
$
|\alpha| = |a_\theta\alpha - b_\theta\beta|.
$
Hence, either $\alpha = a_\theta \alpha - b_\theta\beta$ 
or $\alpha =  b_\theta\beta - a_\theta \alpha$.
Suppose that the former case holds true or, equivalently, that
\begin{equation}\label{eq:proof_Q_S_5}
\frac{\beta}{\alpha} = \frac{a_\theta - 1}{b_\theta}.
\end{equation}
Before proceeding, consider the second expression
for $\overline Q_S^{\,\theta}$ in $\mathcal V_S$ 
given in Remark \ref{remarkQ_S}, namely 
\[
\overline Q_S^{\,\theta}
\,=
c\Big(\frac{\beta^2 - \alpha^2}{\alpha} - \frac{\kk d_\theta}{2c}\Big) + \overline Q_{S,2}^{\,\theta}(\alpha,\beta).
\]
Differentiating it with respect to $\beta$ and setting
$\partial_\beta\overline Q_S^{\,\theta}=0$
yields
\[
2c\Big(\frac{\beta^2 - \alpha^2}{\alpha} - \frac{\kk d_\theta}{2c}\Big)\frac{\beta}{\alpha} = c_1\kk b_\theta - 4c\beta.
\]
This expression, coupled with \eqref{eq:proof_Q_S_3} 
and \eqref{eq:proof_Q_S_5}, easily gives 
\[
(\alpha,\beta)=
\frac\kk2
\left(\frac{b^2}{a_\theta-1},b_\theta\right).
\]
Recalling the definition of $a_\theta$ and $b_\theta$,
this point coincides with 
$(\alpha_\theta^S,\beta_\theta^S)$ defined in the statement,
and lies in the interior of $\mathcal V_S$. 
Supposing now $\alpha =  b_\theta\beta - a_\theta \alpha$
and proceeding similarly returns the point
\[
-\frac{\kk c_2}{2c}
\left(
\frac{b^2_\theta}{a_\theta + 1},
b_\theta
\right),
\]
which belongs to $\mathcal D_S$. 
Hence, the only critical point in the interior of $\mathcal V_S$ 
is $(\alpha_\theta^S,\beta_\theta^S)$.
Other computations show that this is indeed the global minimiser
of $\overline Q_S^{\,\theta}$
and that \eqref{min_Q_S} holds true. 
\end{proof}

In view of the above lemma, we have that the 
minimum of our limiting functional \eqref{funct_lim_S}
is
\begin{equation*}
\min_\mathcal A
\mathscr J_S^{\theta}
\,=\,
\ell\min_{\RR^{2\times2}}\overline Q_S^{\,\theta}
\,=\,
\ell\,\bar e_S
\,=\,
\mu\,\ell\,(1+\bm)
\left(\frac{\pi^4-12}{32}\right)
\frac{\eta_0^2}{h_0^2},
\end{equation*}
where in the second equality we have used \eqref{min_Q_S}
and in the last one the constant $\bar e_S$
has been replaced by its expression given 
in terms of the $3$D parameters
(the first expression in \eqref{bare}).
Also, the minimisers of $\mathscr J_T^{\theta}$ 
are all $(d_1,d_2,d_3)\in\mathcal A$
such that
\begin{align*}
(d_1^\prime\cdot d_3\,,\,d_2^\prime\cdot d_3)
&\,=\,
\frac\kk2
\Big(-1-\cos2\theta ,\sin2\theta 
\Big)\\
&\,=\,
\frac{3\,\eta_0}{\pi^2h_0}
\Big(-1-\cos2\theta,\sin2\theta 
\Big),
\qquad\quad
\theta\in[0,\pi).
\end{align*}
For the convenience of the reader, 
in what follows we enlist some cases:
\begin{align*}
&\theta=0:
\qquad\quad\ 
(d_1^\prime\cdot d_3,d_2^\prime\cdot d_3)
\,=\,
\frac{3\,\eta_0}{\pi^2\,h_0}(-1,0);\\
&\theta=\pi/8:
\qquad\,
(d_1^\prime\cdot d_3,d_2^\prime\cdot d_3)
\,=\,
\frac{3\,\eta_0}{\pi^2\,h_0}
\left(-\frac{\sqrt2+2}2,\frac{\sqrt2}2\right);\\
&\theta=\pi/4:
\qquad\,
(d_1^\prime\cdot d_3,d_2^\prime\cdot d_3)
\,=\,
\frac{3\,\eta_0}{\pi^2\,h_0}(-1,1);\\
&\theta=\pi/2:
\qquad\,
(d_1^\prime\cdot d_3,d_2^\prime\cdot d_3)
\,=\,
(0,0).
\end{align*}



\subsection*{Acknowledgements}
We gratefully acknowledge the support by the European Research Council through the ERC Advanced Grant 340685-MicroMotility.
This work was started after an inspiring lecture given by 
Prof. R. Paroni at ``Physics and Mathematics of Materials: current insights'', an international conference in honour of
the 75th birthday of Paolo Podio-Guidugli
held at Gran Sasso Science Istitute (L'Aquila) in January 2016.


\medskip
\noindent
The authors declare that they have no conflict of interest.


\bibliographystyle{plain}


\end{document}